\documentclass{elsarticle}

\usepackage{hyperref}

\usepackage{caption}
\usepackage{subcaption}
\usepackage{xcolor}
\definecolor{mcol}{RGB}{25,25,255}
\definecolor{tcol}{RGB}{25,200,25}
\definecolor{ccol}{RGB}{228,26,28}
\definecolor{stdred}{RGB}{228,26,28}
\definecolor{stdgreen}{RGB}{26,150,65}
\definecolor{stdblue}{RGB}{31,120,180}
\usepackage{braket} 
\usepackage{xifthen} 
\usepackage{amsmath}
\usepackage{amssymb}
\newtheorem{theorem}{Theorem}
\newtheorem{conjecture}[theorem]{Conjecture}
\newdefinition{definition}[theorem]{Definition}
\newtheorem{lemma}[theorem]{Lemma}
\newtheorem{observation}[theorem]{Observation}
\newtheorem{corollary}[theorem]{Corollary}
\newproof{proof}{Proof}
\usepackage{multibib}
\usepackage{cleveref} 

\newcommand {\calC}{\ensuremath{\mathcal C}}
\newcommand {\calD}{\ensuremath{\mathcal D}}
\newcommand {\calK}{\ensuremath{\mathcal K}}

\newcommand{\NX}[2][]{\ifthenelse{\isempty{#1}}{\ensuremath{N(#2)}}{\ensuremath{N_{#1}(#2)}}}
\newcommand{\EXY}[2][]{\ifthenelse{\isempty{#1}}{\ensuremath{E(#2)}}{\ensuremath{E_{#1}(#2)}}}
\newcommand{\overcirc}[1]{\ensuremath{\mathring{#1}}}
\newcommand{\V}[1]{V(#1)}
\newcommand{\E}[1]{E(#1)}

\usepackage{enumitem}
\newenvironment{enumrom}{
	\begin{enumerate}[labelindent=\parindent, leftmargin=*, label=(\roman*), align=left, noitemsep]
	}{
	\end{enumerate}
}
\newenvironment{enumnum}{
	\begin{enumerate}[labelindent=\parindent, leftmargin=*, label=(\arabic*), align=left, noitemsep]
	}{
	\end{enumerate}
}

\usepackage{tikz}
\usetikzlibrary{calc}
\tikzset{every picture/.style={very thick}}
\newcommand{\cycleSpacing}{120}
\newcommand{\placeCycleCoordinates}[5]%
{%
	\renewcommand{\cycleSpacing}{360/#3}%
	\foreach \x in {1,...,#3}%
	{%
		\node[#5] (#4\x) at ($(\x*\cycleSpacing:#2)+(#1)$) {};%
	}%
}
\newcommand{\placeAndShowCycleCoordinates}[5]%
{%
	\renewcommand{\cycleSpacing}{360/#3}%
	\foreach \x in {1,...,#3}%
	{%
		\node[#5] (#4\x) at ($(\x*\cycleSpacing:#2)+(#1)$) {$\x$};%
	}%
}
\newcommand{\drawCycle}[5]%
{%
	\renewcommand{\cycleSpacing}{360/#3}%
	\foreach \x / \y / \z in #5%
	{%
		\draw[\z,#4] ($(\x*\cycleSpacing:#2)+(#1)$) arc (\x*\cycleSpacing:\y*\cycleSpacing:#2);%
	}%
}

\begin{document}

\begin{frontmatter}
	\title{Towards obtaining a 3-Decomposition from a perfect Matching}
	
	\author[1]{Oliver Bachtler\corref{cor1}}
	\cortext[cor1]{Corresponding author}
	\ead{bachtler@mathematik.uni-kl.de}
	
	\author[1]{Sven O. Krumke}
	\ead{krumke@mathematik.uni-kl.de}
	
	\address[1]{Department of Mathematics, Technische Universit\"at Kaiserslautern, Paul-Ehrlich-Str. 14, 67663~Kaiserslautern, Germany}
	\begin{abstract}	
		A decomposition of a graph is a set of subgraphs whose edges partition those of~$G$.
		The 3-decomposition conjecture posed by Hoffmann-Ostenhof in 2011 states that every connected cubic graph can be decomposed into a spanning tree, a 2-regular subgraph, and a matching.
		It has been settled for special classes of graphs, one of the first results being for Hamiltonian graphs.
		In the past two years several new results have been obtained, adding the classes of plane, claw-free, and 3-connected tree-width~3 graphs to the list.
		
		In this paper, we regard a natural extension of Hamiltonian graphs:
		removing a Hamiltonian cycle from a cubic graph leaves a perfect matching.
		Conversely, removing a perfect matching~$M$ from a cubic graph~$G$ leaves a disjoint union of cycles.
		Contracting these cycles yields a new graph~$G_M$.
		The graph~$G$ is star-like if $G_M$ is a star for some perfect matching~$M$, making Hamiltonian graphs star-like.
		We extend the technique used to prove that Hamiltonian graphs satisfy the 3-decomposition conjecture to show that 3-connected star-like graphs satisfy it as well.	
	\end{abstract}
\end{frontmatter}

\section{Introduction}
\label{sec:intro}
A decomposition of a graph~$G$ is a set of subgraphs such that any edge of $G$ is contained in exactly one of them.
The 3-decomposition conjecture was posed by Hoffmann-Ostenhof in \cite{Ost11} and also appears in BCC22 \cite{Cam11} as Problem~516:
\begin{conjecture}
	Every connected cubic graph has a decomposition consisting of a spanning tree, a 2-regular subgraph, and a matching, which is called a 3-decomposition.
\end{conjecture}
Note that the last two components are allowed to be the empty graph and formally the last component is a subgraph whose edge set is a matching.

The conjecture was proved to be true for connected cubic graphs that are Hamiltonian by Akbari, Jensen, and Siggers \cite{AJS15}.
In 2016, Abdolhosseini et al.\ \cite{AAHM16} showed that traceable is a sufficient requirement already.
Ozeki and Ye~\cite{OY16} proved that 3-connected cubic graphs satisfy the conjecture if they are planar or on the projective plane while Bachstein \cite{Bac15} deals with 3-connected cubic graphs on the Torus and Klein Bottle.
The first of these results was extended to all connected plane graphs by Hoffmann-Ostenhof, Kaiser, and Ozeki \cite{HKO18} in 2018.
In the same year it was also proved to hold for claw-free (sub)cubic graphs by Aboomahigir, Ahanjideh, and Akbari \cite{AAA18}.
More recently, Lyngsie and Merker \cite{LM19} showed that weakening the matching requirement to allow for paths of length~2 suffices to make the conjecture true and Heinrich \cite{Hei19} proved the conjecture for 3-connected cubic graphs of tree-width~3.
Earlier this year, Xie, Zhou, and Zhou \cite{XZZ20} proved its validity when the graph has a two-factor consisting of three cycles.

In this paper we look at graphs that are a natural extension of Hamiltonian cubic graphs in this context.
Notice that a cubic graph $G$ with a Hamiltonian cycle $C$ has a perfect matching, namely the edges of $G-\E{C}$ where $\E{C}$ denotes the edges of $C$.
In general, for a cubic graph $G$ with a perfect matching $M$, $G-M$ is the disjoint union of cycles, leading us to the following definition.
\begin{definition}
	\label{def:star-like}
	Let $G$ be a connected cubic graph with a perfect matching $M$.
	Then $G-M$ consists of disjoint cycles and contracting these in $G$ to single vertices yields a new graph $G_M$, the \emph{contraction graph}, that has a vertex for every cycle in $G-M$ and an edge between two nodes if the corresponding cycles are connected by an edge of $M$.
	If $G$ has a perfect matching $M$ such that $G_M$ is a star (a tree with diameter at most~2), then $G$ is \emph{star-like}.
\end{definition}

We wish the make a few remarks on this definition.
First note that all Hamiltonian cubic graphs are star-like and, by Petersen's theorem \cite{Pet1891}, all bridgeless cubic graphs have a perfect matching.
Since many conjectures in graph theory, a prominent example being the cycle double cover conjecture \cite{Jae85}, consider or can be reduced to bridgeless cubic graphs, obtaining structural information about these is of interest.
Also, using this definition, the main theorem in \cite{XZZ20} now reads that the conjecture is satisfied for any connected cubic graph with a perfect matching such that its contraction graph has order~3.
This extends the previous proofs for Hamiltonian and traceable graphs, which handle contraction graphs of orders~1 and~2.

Our goal is to prove that:
\begin{theorem}
	\label{3DC-3-connected-star-like}
	Every 3-connected star-like graph has a 3-decomposition.
\end{theorem}
The idea of the proof is to construct a tree on the vertices of the centre cycle and to iteratively extend it to the tips of the star.
Once extended to all cycles it yields a 3-decomposition.
To make this precise, we introduce two types of decompositions in the next section.
One describes this tree and the other formalises the properties that we need to extend it to further cycles.
We also show how and prove that the extension works.
In Section~\ref{sec:dec-in-cycles} we present types of decompositions we can find in cycles.
This has striking similarities to the techniques used in \cite{XZZ20}, which we describe in more detail when they occur.
Using these we construct a 3-decomposition of a 3-connected star-like graph in Section~\ref{sec:proof}.
Finally, we note that we can construct graphs of this type that are not in any of the classes for which the theorem has already been proved. 
This construction can be found in Section~\ref{sec:example}.

\section{Decompositions and their Extension}
\label{sec:dec-ext}
The basic notation for this paper is mainly based on \cite{Die10}, but we briefly summarise what we need here.
All graphs are finite and contain neither self-loops nor parallel edges.
The vertex and edge set of a graph $G$ are denoted by $\V{G}$ and $\E{G}$.
For sets $X,Y\subseteq \V{G}$ we write $\NX{X}$ for the set of neighbours of $X$ and $\EXY{X,Y}$ for the edges of $G$ with one end in $X$ and the other in $Y$.
If $Y=\V{G}\setminus X$, we shorten this to $\EXY{X}$.
We write $uv$ for an edge with ends $u$ and $v$.
A path $P$ is a sequence of distinct vertices $v_0v_1\ldots v_k$ such that $v_{i-1}v_i\in \E{G}$ for $i\in\{1,\ldots,k\}$.
By $v_iPv_j$ with $i\leq j$ we denote the subpath $v_i\ldots v_j$.
The notation $\overcirc{v_i}Pv_j$, for $i<j$, describes the subpath $v_{i+1}\ldots v_j$, $v_iP\overcirc{v_j}$ and $\overcirc{v_i}P\overcirc{v_j}$ are defined analogously, and $\overcirc{P}=\overcirc{v_0}P\overcirc{v_k}$.

From now on, let $G$ be a star-like cubic graph with perfect matching $M$ and cycles $C_1,\ldots,C_l$, where $C_1$ is the centre cycle.
We write $\partial(H)$ for the set of degree~2 vertices in $G[\V{H}]$ for a subgraph $H\subseteq G$ and call these vertices the \emph{boundary of $H$}.
For $\emptyset\neq I\subseteq \{1,\ldots, l\}$, we denote $\bigcup_{i\in I} VC_i$ by $V_I$ and $G[V_I]$ by $G_I$, writing $G_i$ for $G_{\{i\}}$.
Recall that a decomposition of a graph is a set of subgraphs such that any edge is contained in exactly one of them.

As promised, we begin with the two types of decompositions we need, starting with the one describing the tree we wish to extend.
Intuitively, it describes a tree $T$ in $G_I$, for $1\in I \subseteq \{1,\ldots,l\}$, that could be part of a 3-decomposition in the entire graph.
The definition does this by ensuring a few necessary conditions:
It requires that all vertices of degree~3 in $G_I$ are part of the tree and that edges not in $T$ should either be matching edges or part of cycles or paths, which is what you obtain when restricting a collection of cycles to a subgraph.
These paths need to be extended to cycles in a later step, so they must end at the boundary $\partial(G_I)$.
\begin{definition}
	\label{def:I-decomposition-tree}
	Let $1\in I \subseteq \{1,\ldots,l\}$ and $\calD_I= \{T_I,\calC_I,(V_I,M_I)\}$ be a decomposition of $G_I$ such that
	\begin{itemize}
		\item $T_I$ is a tree spanning all degree~3 vertices of $G_I$,
		\item $\calC_I$ is a spanning subgraph of $G_I$ with maximum degree~2, and
		\item $M_I$ is a matching.
	\end{itemize}
	If all path components (components that are paths) of $\calC_I$ end at vertices in $\partial(G_I)$, then $\calD_I$ is an $I$-decomposition.
\end{definition}

Note that for $I=\{1,\ldots,l\}$ this is just a 3-de\-com\-po\-si\-tion since $\calC_I$ is no longer allowed to contain path components.
We also remark that this does not describe all possible restrictions, we would have to allow forests for that to be true, but for the upcoming proof trees suffice.

An example of an $I$-decomposition is shown in \Cref{fig:I-decomposition}.
There the edges in $T_I$ are coloured in green, those in $\calC_I$ are red, and the ones in $M_I$ are blue.
This is also our colour scheme for figures throughout this paper.
The edges on the boundary are not actually part of the decomposition, but they exist and their colours describe which component they should eventually end up in.
\begin{figure}[htb]
	\centering
	\begin{subfigure}{.49\textwidth}
		\centering
		\begin{tikzpicture}[scale = 0.7]
		\placeCycleCoordinates{-4,1}{1}{12}{B}{}
		\draw[color = mcol] (B6.center) to[bend left=30] (B11.center);
		\draw[color = tcol] (B3.center) to[bend left=20] (B8.center);
		
		\placeCycleCoordinates{0,0}{2}{20}{A}{}
		\draw[color = mcol] (A10.center) to[bend left=30] (A15.center);
		\draw[color = tcol] (A1.center) to[bend left=30] (A5.center);
		\draw[color = mcol] (A2.center) to ($(A2.center) + (0.4,0.1)$);
		\draw[color = ccol] (A4.center) to[bend left=15] ($(A4.center) + (0.3,0.3)$);
		\draw[color = ccol] (A9.center) to[bend right=15] ($(A9.center) - (0.5,0)$);
		\draw[color = tcol] (A7.center) to[bend left=15] ($(A7.center) + (-0.4,0.1)$);
		\draw[color = tcol] (A17.center) to[bend right=15] ($(A17.center) + (0.4,-0.1)$);
		\draw[color = tcol] (B1.center) to[bend left=30] (A8.center);
		\draw[color = mcol] (B10.center) to[bend right=30] (A11.center);
		\draw[color = tcol] (B7.center) .. controls (-5.5,-1) and (-2,-2.5) .. (A13.center);
		
		\drawCycle{-4,1}{1}{12}{}{{0/2/tcol,2/4/ccol,4/12/tcol}}
		\drawCycle{0,0}{2}{20}{}{{0/4/tcol,4/9/ccol,9/17/tcol,17/18/mcol,18/20/tcol}}
		\draw[color=ccol,cap=round] (B2.center) to[bend left=30] (B4.center);
		\draw[color=tcol,cap=round] (A14.center) to[bend left=30] (A18.center);	
		
		\node[circle, draw, fill=black, label={[yshift=-20pt]$m$}, scale = 0.5] at (A17) {};	
		\node[circle, draw, fill=black, label={[xshift=-3pt]$p$}, scale = 0.5] at (A4) {};	
		\node[circle, draw, fill=black, label={[xshift=-3pt]$p$}, scale = 0.5] at (A9) {};	
		\node[circle, draw, fill=black, label={[xshift=3pt]$2$}, scale = 0.5] at (A2) {};
		\node[circle, draw, fill=black, label={[xshift=-1pt]$0$}, scale = 0.5] at (A7) {};	
		\end{tikzpicture}
		\caption{An $I$-decomposition with two cycles.
		\label{fig:I-decomposition}}
	\end{subfigure}
	\hfill
	\begin{subfigure}{.49\textwidth}
		\centering
		\begin{tikzpicture}[scale=0.7]
		\placeCycleCoordinates{0,0}{2}{20}{A}{}
		\draw[color = mcol] (A12.center) to[bend left=30] (A16.center);
		\draw[color = tcol] (A10.center) to[bend left=30] (A14.center);
		\draw[color = mcol] (A1.center) to[bend left=30] (A5.center);
		\draw[color = ccol] (A9.center) to[bend right=15] ($(A9.center) - (0.5,0)$);
		\draw[color = ccol] (A19.center) to[bend left=15] ($(A19.center) + (0.5,0)$);
		\draw[color = tcol] (A4.center) to[bend left=15] ($(A4.center) + (0.3,0.3)$);
		\draw[color = tcol] (A7.center) to[bend left=15] ($(A7.center) + (-0.4,0.1)$);
		\draw[color = tcol] (A17.center) to[bend right=15] ($(A17.center) + (0.4,-0.1)$);
		\draw[color = mcol] (A2.center) to ($(A2.center) + (0.4,0.1)$);
		\drawCycle{0,0}{2}{20}{}{{0/9/tcol,9/11/ccol,11/18/tcol,18/19/ccol,19/20/tcol}}	
		\draw[color = ccol,cap=round] (A11.center) to[bend left=30] (A18.center);
		\node[circle, draw, fill=black, label={[xshift=-3pt]$p$}, scale = 0.5] at (A9) {};	
		\node[circle, draw, fill=black, label={[xshift=3pt,yshift=-20pt]$p$}, scale = 0.5] at (A19) {};	
		\node[circle, draw, fill=black, label={[]$m$}, scale = 0.5] at (A7) {};	
		\node[circle, draw, fill=black, label={[xshift=3pt]$2$}, scale = 0.5] at (A2) {};	
		\node[circle, draw, fill=black, label={[yshift = -20pt]$2$}, scale = 0.5] at (A17) {};	
		\node[circle, draw, fill=black, label={[xshift=-1pt]$0$}, scale = 0.5] at (A4) {};
		\end{tikzpicture}
		\caption{An $(A_0,A_p,A_m,A_2)$-decomposition.
		\label{fig:a0pm2-decomposition}}
	\end{subfigure}
	\caption{Sketches of the two types of decompositions introduced in Definitions~\ref{def:I-decomposition-tree} and~\ref{def:a0pm2-decomposition-tree}.
	Green edges should be part of the tree in the final decomposition, red ones are on cycles, and the blue ones form a matching.
	\label{fig:I-a0pm2-decomposition}}
\end{figure}
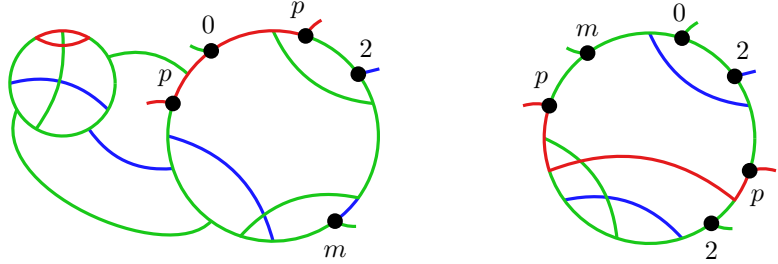

A bit of additional notation is useful at this point.
We write $A_i(\calD_I)$ for the set of vertices in $\partial(G_I)$ that have degree $i$ in $T_I$, where $A_0(\calD_I)$ denotes those that are not in $T_I$ at all.
Moreover, we split the set $A_1(\calD_I)$ into $A_p(\calD_I)$ and $A_m(\calD_I)$, where the former contains those degree~1 vertices of $T_I$ that are ends of path components of $\calC_I$, whereas the latter contains the ends of matching edges.
As the vertices in $A_1(\calD_I)$ have degree~2 in $G_I$, these are the only two possibilities.
As a result, we have that $\partial(G_I)$ is the disjoint union of the four sets $A_x(\calD_I)$ for $x\in\{0,p,m,2\}$.
The drawn vertices in \Cref{fig:I-decomposition} represent the boundary and those in set $A_x(\calD_I)$ are labelled by $x$.

We can now move on to the second type of decomposition we need.
The next definition might seem cryptic at first glance, but it, like the previous one, is essentially just a collection of necessary conditions.
Our goal is to formalise the extension of an $I$-decomposition $\calD_I$ to another cycle $C_i$ by describing a spanning forest $F_i$ that satisfies conditions analogous to those of the tree $T_I$ above.
Once again, the remaining edges should either be part of a matching $M_i$ or on cycles or paths in $\calC_i$.

But it also needs to ``fit together'' with the $I$-decomposition.
We notice that we do not actually need the details of $\calD_I$, but it suffices to know the behaviour of the vertices on the boundary of $G_I$ with an edge to $C_i$.
Let $u$ be such a vertex with unique neighbour $v$ in $C_i$, more precisely in $\partial(C_i)$.
It can exhibit different types of behaviour:

If $u\in A_2(\calD_I)$, then $uv$ can either be in the tree or matching part, depending on whether we need it to connect to $C_i$ or not.
Here, $v$ goes into the set $A_2$.

When $u$ is in $A_p(\calD_I)$, we need to extend the path ending at this vertex to a cycle, meaning it must continue in $C_i$.
Hence, we require a path at $v$ to another vertex in $\partial(C_i)$ of this type.
We put $v$ into the set $A_p$ and require that the vertices in this set are exactly the ends of path components of $\calC_i$.
This is a necessary condition as paths must end at the boundary and all other types of vertices have conflicting behaviour.

If $u$ is in $A_m(\calD_I)$, then $v$ goes into $A_m$.
In this case, the edge $uv$ must be part of the tree and we have to ensure that it creates no cycles.
This is achieved by requiring that any component of $F_i$ contains at most one vertex that is in $A_m$ or in $A_2$ and a leaf of $F_i$.
Such vertices need an edge to $T_I$, which we have already seen for those in $A_m$ and it holds for the leaves as well:
The missing edge at such a vertex must be in the matching as the ends of paths are in $A_p$.

Finally, $u$ can be in $A_0(\calD_I)$, where it needs the edge $uv$ to be part of the tree and $v$ must be connected to $T_I$ in~$C_i$.
Now $v$ is put in the set $A_0$ and we ensure that it ends up in a component of $F_i$ that can be connected to $T_I$.
For this we require every component to contain an element of $A_2\cup A_m$, which are vertices that can or need to connect to $T_I$.
With these ideas at hand, let us give the definition.
\begin{definition}
	\label{def:a0pm2-decomposition-tree}
	Let $A_0$, $A_p$, $A_m$, and $A_2$ be disjoint subsets of $\partial(C_i)$ for some $i\in\{1,\ldots,l\}$ whose union is $\partial(C_i)$.
	Also let $\calD_i = \{F_i,\calC_i,(V_i,M_i)\}$ be a decomposition of $G_i$ such that
	\begin{itemize}
		\item $F_i$ is a spanning forest in $G_i$,
		\item $\calC_i$ is a spanning subgraph of $G_i$ with maximum degree~2, and
		\item $M_i$ is a matching.
	\end{itemize}
	The decomposition $\calD_i$ is an \emph{$(A_0,A_p,A_m,A_2)$-decomposition of $C_i$} if
	\begin{enumrom}
		\item for every component $K$ of $F_i$, the set $\V{K}\cap (A_2\cup A_m)$ is non-empty and contains at most one vertex that is a leaf of $F_i$ or contained in $A_m$ and\label{def:a0pm2-dec-cond-F}
		\item the set of ends of path components of $\calC_i$ is exactly $A_p$.\label{def:a0pm2-dec-cond-Ci}
	\end{enumrom}
\end{definition}

\Cref{fig:a0pm2-decomposition} visualises such a decomposition, where vertices in $A_x$ are labelled by $x$ for $x\in\{0,p,m,2\}$ and the colour scheme is analogous to before: 
the edges in $F_i$ are coloured in green, those in $\calC_i$ are red, and the ones in $M_i$ are blue.

We are now in a position to prove that the necessary conditions we incorporated into our definitions suffice to let us extend an $I$-decomposition.
More precisely, we show that an $I$-decomposition can be extended to a new cycle $C_i$ if we have an $(A_0,A_p,A_m,A_2)$-decomposition there.
Here the sets $A_x$, for $x\in\{0,p,m,2\}$, are assigned exactly those vertices we gave them when we described the intuition behind Definition~\ref{def:a0pm2-decomposition-tree}.
After this, we only need to take a look at what kinds of decompositions we can find in the cycles $C_i$ and how we can piece them together to obtain one of $G$.
\begin{lemma}
	\label{expanding-I-decompositions}
	Let $1\in I \subseteq \{1,\ldots,l\}$, $i\notin I$, $J=I\cup\{i\}$, and $\calD_I = (T_I,\calC_I,M_I)$ be an $I$-de\-com\-po\-si\-tion of $G$.
	If an $(A_0,A_p,A_m,A_2)$-decomposition $\calD_i = (F_i,\calC_i,M_i)$ of $C_i$ exists where
	\begin{displaymath}
	A_x = \NX{A_x(\calD_I)}\cap \V{C_i} \text{ for } x\in\{0,p,m,2\},
	\end{displaymath}
	then $G$ has a $J$-decomposition $(T_J,\calC_J,M_J)$ with $T_I\cup F_i\subseteq T_J$, $\calC_I\cup \calC_i\subseteq \calC_J$, and $M_I\cup M_i\subseteq M_J$.
\end{lemma}
\begin{proof}
	Let $\calD_I$ and $\calD_i$ be decompositions as in the claim.
	In order to get a decomposition $\calD_J$ as desired, we need to assign the edges in $\EXY{\V{G_I},\V{C_i}}$ to the graphs $T_I\cup F_i$, $\calC_I\cup \calC_i$, and the set $M_I\cup M_i$.
	To this end, we proceed as follows.
	Note that, by definition of the sets $A_x$, $\EXY{\V{G_I},\V{C_i}} = \EXY{\partial(G_I),\V{C_i}} = \bigcup_x\EXY{A_x(\calD_I),A_x}$ where $x\in\{0,p,m,2\}$.
	We add the set $\EXY{A_p(\calD_I),A_p}$ to $\calC_I\cup\calC_i$ to get $\calC_J$.
	The sets $\EXY{A_m(\calD_I),A_m}$ and $\EXY{A_0(\calD_I),A_0}$ are both added to $T_I\cup F_i$.
	Additionally, for any component $K$ of $F_i$ that contains a vertex of $A_2$ but none of $A_m$, we pick a vertex from~$A_2\cap \V{K}$ of least degree in $K$ and add the edge incident to it with end in $A_2(\calD_I)$ to the tree part as well.
	(Such vertices exist by Condition~\ref{def:a0pm2-dec-cond-F} and the minimality just means that we choose a leaf in case one is present.)
	This yields~$T_J$.
	The remaining edges of $\EXY{A_2(\calD_I),A_2}$ are added to $M_I\cup M_i$ to get $M_J$.
	
	We claim that $\calD_J= (T_J,\calC_J,M_J)$ is a desired $J$-decomposition of $G$.
	The set $\calC_J$ is the union of two disjoint graphs $\calC_I$ and $\calC_i$ of maximum degree~2 together with edges $\EXY{A_p(\calD_I),A_p}$ connecting degree~1 vertices of these subgraphs.
	Hence it, too, has maximum degree~2 as required.
	Furthermore, a degree~1 vertex in $\calC_J$ must have degree~1 in $\calC_I$ or $\calC_i$.
	In the first case it is an element of $\partial(G_I)$ by definition of an $I$-decomposition and it cannot be part of $\NX{C_i}$ without increasing its degree when we add the edges in \EXY{A_p(\calD_I),A_p}.
	So it is in $\partial(G_J)$ as desired.
	The second case does not occur as vertices of degree~1 in $\calC_i$ are in $A_p$ and have degree 2 in $\calC_J$.
	
	The set $M_J$ is also a matching as it is the union of two matchings $M_I$, $M_i$ in disjoint subgraphs and the additional edges are part of $\EXY{A_2(\calD_I),A_2}$, meaning their ends in $G_I$ have degree~2 in $T_I\subseteq T_J$.
	Their ends in $C_i$ also have degree~2 in $F_i\subseteq T_J$ as a lower degree makes them a leaf or isolated vertex of $F_i$.
	In the first case the component containing that vertex cannot contain a vertex in $A_m$ and the leaf is unique, meaning the edge is added to $T_J$ by our construction.
	The second case faces a component with a unique edge to $G_I$, which is also added to $T_J$.
	
	This just leaves $T_J$.
	Let \calK{} be the set of components of $F_i$ and let $F$ be the union of $T_I$ with the components in \calK.
	By adding the edges of $\EXY{A_m(\calD), A_m}$ to $F$ we have connected all components $K\in\calK$ that contain a vertex of $A_m$ to $T_I$ by exactly one edge each.
	The result is a new forest $F'$ consisting of a tree $T'\supseteq T_I$ and remaining components $\calK'\subseteq\calK$ that have no vertex in $A_m$.
	By adding our chosen elements of $\EXY{A_2(\calD_I),A_2}$ we connect the components of $\calK'$ (as these contain an element of $A_2$) to $T'$ by exactly one edge.
	This results in a tree $T''$.
	Finally, the last missing edges in $\EXY{A_0(\calD_I),A_0}$ connect vertices of $G_I$ that are not in $T''$ to it by a single edge, creating the tree $T_J$.
	
	Now we only need to check that $T_J$ spans all vertices of degree~3 in $G_J$.
	To this end regard a vertex of $G_J$ that is not part of $T_J$.
	It cannot be in $C_i$ as all the components of $F_i$ are part of $T_J$ and $F_i$ was spanning.
	A vertex in $G_I$ that is not part of $T_J$ is also not in $T_I$, putting it in $A_0(\calD_I)$.
	But such vertices still have degree~2 in $G_J$ as they cannot be in~$A_0(\calD_I)\cap\NX{C_i}$ because the degree of such a vertex is~1 now.
	\qed
\end{proof}

\section{Finding Decompositions in Cycles}
\label{sec:dec-in-cycles}
In this section, let $C_i$ be some cycle in $G-M$ and $A_0$, $A_p$, $A_m$, and $A_2$ be disjoint subsets of $\partial(C_i)$ for which we want to find an $(A_0,A_p,A_m,A_2)$-decomposition.
We need four different types of decompositions in order to handle all cases that occur when piecing them together.

Before we start, a bit more notation will come in handy, that we now introduce.
For a chord $e$ of $C_i$ we obtain two paths in $C_i$ between its ends which, together with the chord, yield two cycles, say $C_i'$ and $C_i''$.
We call $C\in \{C_i', C_i''\}$ \emph{minimal} if it is a chordless cycle in $G$.
The unique edge in $M\cap \E{C}$ of a minimal cycle $C$ is denoted by $e_C$, and we write $P_C$ for the path $C-e_C$.
Note that $C_i$ has a minimal cycle avoiding any specific vertex in $\partial(C_i)$ if it has a chord.
To see this, take a chord $e$ with cycles $P_1+e$ and $P_2+e$.
By choosing $v_j,w_j$ as ends of an edge in $M$ of minimal distance in $P_j$, for $j\in\{1,2\}$, we find two minimal cycles $v_jP_jw_jv_j$ that meet disjoint sets of vertices of $\partial(C_i)$.
Note that the vertices $v_j,w_j$ always exist as the ends of $e$ are candidates.
The cycles are minimal as a chord of $v_jP_jw_jv_j$ must be an edge of $M$ whose ends have smaller distance.

A useful construction that we apply regularly is the following.
Let $C$ be a minimal cycle in $C_i$ that does not contain some vertex $x\in A_2\cup A_m$ and where $\V{C}\cap \partial(C_i)$ contains only vertices of $A_2$.
We assign the edges of $\E{G_i}$ to our three components by setting $\calC_i = (\V{C_i},\E{C})$ and $M' = M\cap \E{G_i}\setminus \E{C}$, $F'=C_i-\E{C}$.
In this assignment $\calC_i$ has maximum degree~2 and contains no path components, $M'$ is a matching, and $F'$ consists of a path $P$ together with a set of isolated vertices.
As $x\in P$, this path is not disjoint from $A_2\cup A_m$ and has no leaf in $\partial(C_i)$ as its ends are incident to a chord, so it satisfies~\ref{def:a0pm2-dec-cond-F}.
Let $v$ be an isolated vertex of $F'$.
If $v$ has degree~3 in $G_i$, then it is incident to an edge $vu\notin C$ whose other end is in $P$.
We remove this edge from $M'$ and add it to $F'$, leaving $F'$ acyclic by adding a new leaf to the tree $P$.
As this leaf has degree~3, the larger component continues to satisfy Property~\ref{def:a0pm2-dec-cond-F}.
In the case where $v$ has degree~2, we assumed that $v\in A_2$ and this component also satisfies~\ref{def:a0pm2-dec-cond-F}.
The resulting spanning forest $F_i$ and matching $M_i$ therefore form an $(A_0,A_p,A_m,A_2)$-decomposition of $C_i$ if $A_p=\emptyset$. 
We call this the \emph{decomposition given by $C$}.

We now show that certain $(A_0,A_p,A_m,A_2)$-decompositions exist, starting with $A_0 = A_p = \emptyset$, $A_m=\{x\}$ for some~$x\in \partial(C_i)$, and $A_2=\partial(C_i)\setminus A_m$.
\begin{lemma}
	\label{eexA-decompositions}
	There exists an $(\emptyset,\emptyset,\{x\},A_2)$-decomposition of $C_i$.
\end{lemma}
\begin{proof}
	If the cycle $C_i$ is chordless, all its vertices are in $\partial(C_i)$ and $\E{G_i} = \E{C_i}$.
	Here setting $F_i=(\V{C_i},\emptyset)$, $\calC_i = C_i$, and $M_i=\emptyset$ does the trick.
	
	In the case where the cycle $C_i$ has a chord, it contains a minimal cycle $C$ that avoids $x$ and we can use the decomposition given by $C$. 
	\qed
\end{proof}

We also find decompositions this way when all elements of $\partial(C_i)$ are in $A_2$.
\begin{corollary}
	\label{eeeA-decompositions}
	If $\partial(C_i)\neq\emptyset$, then $C_i$ has an $(\emptyset,\emptyset,\emptyset,\partial(C_i))$-decomposition.
\end{corollary}
\begin{proof}
	By Lemma~\ref{eexA-decompositions} there exists an $(\emptyset,\emptyset,\{x\},A_2\setminus\{x\})$-decomposition of $C_i$ for some arbitrary $x\in A_2$.
	This is an~$(\emptyset,\emptyset,\emptyset,A_2)$-decomposition by definition. 
	\qed
\end{proof}

Next, let $A_0=\{x\}$ for some $x\in \partial(C_i)$, $A_2=\partial(C_i)\setminus A_0$, and $A_m = A_p = \emptyset$.
\begin{lemma}
	\label{xeeA-decompositions}
	If $A_2\neq\emptyset$, then there exists an $(\{x\},\emptyset,\emptyset,A_2)$-decomposition of $C_i$.
\end{lemma}
\begin{proof}
	We begin by looking at the case where $C_i$ is chordless.
	Let $y$ be a neighbour of $x$ in $C_i$ and regard the spanning tree $F_i=C_i-xy$.
	This contains an element of $A_2$ and it has only one leaf in $A_2$, namely $y$.
	Thus, $F_i$ satisfies Property~\ref{def:a0pm2-dec-cond-F}.
	The last missing edge $xy$ of $\E{G_i}$ is assigned to $M_i$, making this a matching and leaving $\calC_i$ with no edges and thus no path component.
	This gives us an $(\{x\},\emptyset,\emptyset,A_2)$-decomposition.
	
	Now we assume that $C_i$ has a chord.
	The existence of a minimal cycle $C$ that neither contains $x$ nor all elements of $A_2$ is another good case as it satisfies both requirements necessary for us to obtain a decomposition given by $C$.
	
	In the final and most complicated case, we may assume that $C$ has chords but none of them yield a cycle as described above.
	We have already seen that any chord naturally gives rise to two minimal cycles $C,\ C'$ for which~$P_C,\ P_{C'}$ have no inner vertex in common.
	Consequently one of them must contain $x$ while the other contains all vertices of $A_2$.
	Hence, all vertices of $A_2$ must form a path~$P$:
	A vertex of degree~3 between them is incident to a chord and this would yield a minimal cycle as above.
	Let $P_1$ and $P_2$ be the two paths in $C_i-\E{P}$ between $x$ and the ends of~$P$.
	Then any chord $uv$ of $C_i$ must connect an inner vertex of $P_1$ to one of $P_2$:
	If both were on the same path, then one of the two minimal cycles that $uv$ yields would contain neither $x$ nor any element of $A_2$, contradicting our assumption.
	Hence, $P_1$ and $P_2$ have the same length.
	
	Let $x_1,\ldots,x_r$ and $y_1,\ldots,y_r$ be the inner vertices of $P_1$ and $P_2$ respectively, ordered by increasing distance to $x$.
	We call a chord $x_ky_l$ of $C_i$ \emph{short} if $|k-l|\leq 1$ and \emph{long} otherwise.
	As an illustration, the two blue chords in \Cref{fig:xeeA-decompositions-long-chord} are short, whereas the edge $x_ky_l$ is a long chord.
	It turns out that long chords are helpful and the presence of only short ones is structurally very restrictive.
	More precisely, we inductively prove that if all chords $x_ky_l$ in $G_i$ with $k,l\leq d$ are short, then such a chord $x_ky_l$ is in $G$ if and only if $x_ly_k$ is.
	The case that $d\leq 1$ is clear as the only candidate is $x_1y_1$.
	So let it hold up to $d-1$ for $d\geq 2$ and let $x_ky_l$ satisfy $k,l\leq d$.
	If $k,l<d$ or $k,l = d$ we are done, so we may assume, by symmetry, that $k=d$ and $l=d-1$.
	But $x_{d-1}$ is matched to a vertex in $\{y_{d-2},y_{d-1},y_d\}$ by $M$.
	Of these three, only $y_d$ is an option as $y_{d-1}$ is taken by $x_d$ and $x_{d-1}y_{d-2}$ cannot be in $M$ by induction since $x_{d-2}y_{d-1}$ is not.
	
	If a long chord $x_ky_l$ exists, choose one minimising $\min\{k,l\}$.
	By symmetry, we can assume that $k<l$ and all chords with an end of index at most $k-1$ are short.
	As a result, the vertices $y_k,y_{k+1}$ are matched to vertices in $x_{k+1}P_1$.
	This holds as neither is matched to $x_k$ whose neighbour in $G[M]$ is $y_l$ with $l\geq k+2$.
	All vertices in $P_1x_{k-2}$ have neighbours in $P_2y_{k-1}$, so none of these are possible either.
	This just leaves the vertex $x_{k-1}$.
	Since $x_ky_{k-1}\notin M$, $y_{k-1}$ is matched to $x_{k-1}$ or $x_{k-2}$, giving us $x_{k-1}y_{k-1}\in M$ or $x_{k-1}y_{k-2}\in M$, eliminating $x_{k-1}$ as well.
	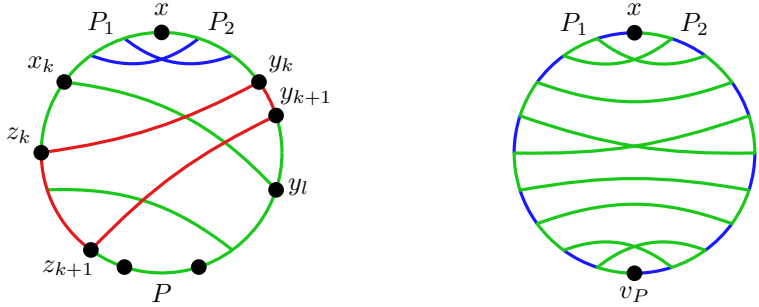
\begin{figure}[htb]
		\centering
		\begin{subfigure}{.49\textwidth}
			\centering
			\begin{tikzpicture}[scale=0.8]
				\placeCycleCoordinates{0,0}{2}{20}{A}{}
				\draw[color = mcol] (A3.center) to[bend left] (A6.center);
				\draw[color = mcol] (A4.center) to[bend left] (A7.center);
				\draw[color = tcol] (A8.center) to[bend left=20] (A19.center);
				\draw[color = tcol] (A17.center) to[bend right=20] (A11.center);
				\drawCycle{0,0}{2}{20}{}{{0/1/tcol,1/2/ccol,2/10/tcol,10/13/ccol,13/20/tcol}}
				\node[circle, draw, fill=black, label={[xshift=0mm, yshift=0mm]$x$}, scale = 0.5] at (A5) {};
				\node[label={[xshift=0mm, yshift=-6mm]$P$}, scale = 0.5] at (A15) {};
				\node[circle, draw, fill=black, scale = 0.5] at (A14) {};
				\node[circle, draw, fill=black, scale = 0.5] at (A16) {};
				\draw[color = ccol,cap=round] (A2.center) to[bend left=10] (A10.center);
				\draw[color = ccol,cap=round] (A1.center) to[bend right=10] (A13.center);
				\node[circle, draw, fill=black, label={[xshift=3mm, yshift=-1mm]$y_k$}, scale = 0.5] at (A2) {};
				\node[circle, draw, fill=black, label={[xshift=4mm, yshift=-1mm]$y_{k+1}$}, scale = 0.5] at (A1) {};
				\node[circle, draw, fill=black, label={[xshift=3mm, yshift=-3mm]$y_{l}$}, scale = 0.5] at (A19) {};
				\node[circle, draw, fill=black, label={[xshift=-3mm, yshift=-1mm]$x_k$}, scale = 0.5] at (A8) {};
				\node[circle, draw, fill=black, label={[xshift=-3mm, yshift=-1mm]$z_k$}, scale = 0.5] at (A10) {};
				\node[circle, draw, fill=black, label={[xshift=-3mm, yshift=-6mm]$z_{k+1}$}, scale = 0.5] at (A13) {};
				\node (P1) at (-1,2.15) {$P_1$};
				\node (P2) at (1,2.15) {$P_2$};
			\end{tikzpicture}
			\caption{The cycle obtained from the long chord $x_ky_l$.
			\label{fig:xeeA-decompositions-long-chord}}
		\end{subfigure}
		\hfill
		\begin{subfigure}{.49\textwidth}
			\centering
			\begin{tikzpicture}[scale=0.8]
				\placeCycleCoordinates{0,0}{2}{20}{A}{}
				\drawCycle{0,0}{2}{20}{}{{0/1/tcol,1/2/mcol,2/3/tcol,3/4/mcol,4/5/tcol,5/6/mcol,6/7/tcol,7/8/mcol,8/9/tcol,9/10/mcol,10/11/tcol,11/12/mcol,12/13/tcol,13/14/mcol,14/15/tcol,15/16/mcol,16/17/tcol,17/18/mcol,18/19/tcol,19/20/mcol}}
				\draw[color = tcol,cap=round] (A3.center) to[bend left] (A6.center);
				\draw[color = tcol,cap=round] (A4.center) to[bend left] (A7.center);
				\draw[color = tcol,cap=round] (A2.center) to[bend left=20] (A8.center);
				\draw[color = tcol,cap=round] (A1.center) to[bend left=10] (A10.center);
				\draw[color = tcol,cap=round] (A20.center) to[bend left=10] (A9.center);
				\draw[color = tcol,cap=round] (A11.center) to[bend left=10] (A19.center);
				\draw[color = tcol,cap=round] (A12.center) to[bend left=20] (A18.center);
				\draw[color = tcol,cap=round] (A13.center) to[bend left] (A16.center);
				\draw[color = tcol,cap=round] (A14.center) to[bend left] (A17.center);
				\node[circle, draw, fill=black, label={[xshift=0mm, yshift=0mm]$x$}, scale = 0.5] at (A5) {};
				\node[circle, draw, fill=black, label={[xshift=0mm, yshift=-6.125mm]$v_P$}, scale = 0.5] at (A15) {};
				\node (P1) at (-1,2.15) {$P_1$};
				\node (P2) at (1,2.15) {$P_2$};
			\end{tikzpicture}
			\caption{The decomposition when all chords are short.
			\label{fig:xeeA-decompositions-short chords}}
		\end{subfigure}
		\caption{Illustrations of the decompositions used in the proof Lemma~\ref{xeeA-decompositions}, using the presence or non-existence of long chords in the cycle $C_i$.}
	\end{figure}
	
	Now, let $z_k,\ z_{k+1}$ be the neighbours of $y_k,\ y_{k+1}$ in $G[M]$ and take a look at the cycle $C = y_ky_{k+1}z_{k+1}P_1z_ky_k$ shown in \Cref{fig:xeeA-decompositions-long-chord}.
	We now apply a construction similar to the one for minimal cycles:
	Assign the edges of $\E{G_i}$ to the three components by setting $\calC_i = (\V{C_i},\E{C})$ and $M' = M\cap \E{G_i}\setminus \E{C}$, $F'=C_i-\E{C}$.
	Then $\calC_i$ has maximum degree~2 and no path components, $M'$ is a matching, and $F'$ consists of two paths together with isolated vertices.
	The ends of these paths are part of $C$, so they have degree~3 in $G_i$ and we can connect the two paths using the long chord $x_ky_l$.
	This replaces the two paths by a tree $T'$ and the isolated vertices have degree~3 in $G_i$ with a neighbour in $P_2$.
	As their neighbours are not on $C$, they are part of $T'$ and we can connect them by adding such edges to $T'$.
	Give $F_i$ the edges of $T'$ and put the remaining edges into $M_i$, then $M_i\subseteq M'$ is still a matching and $F_i$ is a spanning tree.
	Since $F_i$ contains $P$ and thus (all) vertices of $A_2$, the conditions of an $(\{x\},\emptyset,\emptyset,A_2)$-decomposition are satisfied.
	
	This just leaves the case that all chords are short, which makes use of the very potent knowledge of the way these behave.
	Here we give $\calC_i$ no edges of $\E{G_i}$, instead dividing them up amongst $F_i$ and $M_i$.
	Note that the degree~2 vertices of $G_i$ are those in the set $\{x\}\cup A_2$ where the vertices of $A_2$ are all on $P$.
	We suppress all vertices of degree~2 except $x$ and one element of $A_2$.
	Here, suppressing a degree~2 vertex $v$ means removing it and the edges $uv,\ vw$ to its neighbours and adding the direct edge $uw$ between these.
	The resulting graph $G'$ is left with just two vertices of degree~2, namely $x$ and a vertex $v_P$ that has replaced the entire path $P$.
	The paths $P_1$ and $P_2$ are now $xv_P$-paths, so we have $P_1 = xx_1\ldots x_rv_P$ and $P_2 = xy_1\ldots y_rv_P$.
	An illustration of this and the decomposition we now choose can be found in \Cref{fig:xeeA-decompositions-short chords}.
	
	Since $P_1$ and $P_2$ have the same length, $P_1\cup P_2$ is a Hamiltonian cycle of even length.
	Its edges thus decompose into two perfect matchings $M_1$ and $M_2$, letting us obtain a Hamiltonian $xv_P$-path $Q = G' - M_1$ in $G'$.
	To see that this is indeed the case, notice that all vertices of $G' - \{x,v_P\}$ have degree~2 in $Q$, where the two excluded ones have degree~1.
	Hence, $Q$ consists of an $xv_P$-path and possibly additional cycles.
	Suppose it contains a cycle $C$ and choose a vertex of minimal index in $C$.
	By symmetry, we assume this vertex is $x_k\in P_1$.
	Then the edge $x_kx_{k-1}$ is part of $M_1$ as $x_{k-1}\notin C$ (where $x_0,y_0 = x$ and $x_{r+1},y_{r+1} = v_P$).
	Thus, $x_kx_{k+1}$ and $y_ky_{k-1}$ are in $M_2$ and $Q$.
	The edge $x_ky_{k-1}$ cannot be in $G'$ either since $y_{k-1}$ is not part of $C$.
	But now $x_ky_k\in Q$ or $x_ky_{k+1}, x_{k+1}y_k\in Q$.
	Both yield an $x_ky_{k-1}$-path in $C$, $x_ky_ky_{k-1}$ or $x_kx_{k+1}y_ky_{k-1}$ respectively, a contradiction.
	
	Hence, $Q$ is a Hamiltonian path and we can obtain a 3-decomposition by replacing $v_P$ by the path $P$ again and putting all edges of $P\cup Q$ into $F_i$.
	Then $F_i$ is a Hamiltonian path in $G_i$, which ends at $x$ and at an end of $P$.
	This is an $(\{x\},\emptyset,\emptyset,A_2)$-decomposition of $C_i$ since the only component of $F_i$ contains an element of $A_2$ and it only has one leaf in $A_2$.
	\qed
\end{proof}

Here we remark that this lemma was also obtained by Xie, Zhou, and Zhou and can be found in \cite[Lemma~2.3]{XZZ20}.
Their formulation basically describes the two cases in the proof, as they claim to either get a decomposition containing a cycle with two chords or a Hamiltonian path.
We repeated the statement to make it fit into our notation.
The reason we also presented our proof is that it is different and we believe that it reveals more structure.
Our case distinction was based on the existence of long chords and we obtained that either the graph has one or all chords are short, in which case it has a Hamiltonian path.
But we also know that all chords in this case are either of the form $x_iy_i$ or they come in a pair $x_iy_{i+1}, x_{i+1}y_i$.
Xie, Zhou, and Zhou prove this by distinguishing whether or not the cycle has a non-separating two-chord cycle.
This turns out to be exactly our distinction, as we obtain such cycles in the case that there is a long chord and they do not exist when all chords are short, but it obscures the structure of the chords.
They also construct a different Hamiltonian path as a result.

Lastly, we let $A_p=\{x,y\}$ for $x,y\in \partial(C_i)$ and $A_2=\partial(C_i)\setminus A_p$.
Due to the abundance of indices needed in the proof and the lack of cycles therein, we denote the regarded cycle by $C$ instead of $C_i$ and write $G_C$ for $G[\V{C}]$.
\begin{lemma}
	\label{exyeA-decompositions}
	If $G$ is 3-connected, then $C$ has an $(\emptyset,\{x,y\},\emptyset,A_2)$-de\-com\-po\-si\-tion.
\end{lemma}
\begin{proof}
	Let $P,P'$ be the two $xy$-paths in $C$.
	Since $G-\{x,y\}$ is connected, $A_2$ is non-empty and we may assume that, without loss of generality, $\V{P'}\cap A_2\neq\emptyset$.
	Next, let $u_1v_1,u_2v_2,\ldots,u_sv_s$ be a maximal sequence of edges of $M$ with ends in $P$ satisfying, for all $i\in\{1,\ldots,s\}$, that
	\begin{enumnum}
		\item the path $P_i=u_iPv_i$ is disjoint from all previous ones, meaning that $P_i\subseteq P-\bigcup_{k<i} P_k$, \label{prop:chord-yield-disjoint-paths}
		\item $P_i$ either contains an element of $A_2$ or there is an edge $e_i\in \EXY{\V{P_i},X_i}$ where $X_i = \V{P'}\cup\bigcup_{k<i}\V{P_k}$, and \label{prop:A2-or-good-neighbour}
		\item $P$ has no vertices $u',v'$ with $u'v'\in M$ and $P_i\subsetneq u'Pv'\subseteq P-\bigcup_{k<i} P_k$. \label{prop:maximal-chords}
	\end{enumnum}

	We remark that these paths $P_i$ end up as part of the tree and Property~\ref{prop:A2-or-good-neighbour} just ensures that they either contain an element of $A_2$, and can form a component, or can connect to a prior path or $P'$, which will have such a vertex by induction or our assumption.
	Also notice that Property~\ref{prop:maximal-chords} can be read as: ``$P_i$ is chosen maximally'', in the sense that it forbids the existence of candidates $u',v'$ for $u_i,v_i$ that would yield a longer path when picked instead.
	
	Let $X_u=\Set{u_i : i=1,\ldots,s}$, $X_v=\Set{v_i : i=1,\ldots,s}$.
	We assume that $u_i$ occurs before $v_i$ in $P$ for all $i$.
	By Property~\ref{prop:chord-yield-disjoint-paths}, no $P_i$ contains an element of $X_u\cup X_v$ as an inner vertex, meaning vertices of $X_u$ and $X_v$ alternate. 
	Now let $y_i$ $(x_i)$ be the $i$th occurrence of a vertex of $X_u$ $(X_v)$ in $P$, for $i\in\{1,\ldots,s\}$, which is just a labelling of the vertices in the order they appear on the path.
	We refer to \Cref{fig:exyeA-decompositions} to keep track of the notation.
	\begin{figure}[htb]
		\centering
		\begin{tikzpicture}[scale=0.85]
			\placeCycleCoordinates{0,0}{2.5}{30}{A}{}
			\draw[color = tcol] (A22.center) to[bend right=30] (A16.center);
			\draw[color = tcol] (A5.center) to[bend right=40] (A1.center);
			\draw[color = ccol] (A24.center) to[bend right=40] (A21.center);
			\draw[color = ccol] (A2.center) to[bend right=40] (A29.center);
			\draw[color = ccol] (A8.center) to[bend right=50] (A4.center);
			
			\draw[color = tcol] (A14.center) to ($(A14.center) + (-0.5,0)$);
			\draw[color = tcol] (A7.center) to ($(A7.center) + (0.1,0.45)$);
			
			\drawCycle{0,0}{2.5}{30}{}{{0/2/tcol,2/4/ccol,4/8/tcol,8/11/ccol,11/18/tcol,18/21/ccol,21/24/tcol,24/25/ccol,25/28/dashed,28/29/ccol,29/30/tcol}}
			
			\node[circle, draw, fill=black, label=above left:{$x$}, scale = 0.5] at (A11) {};
			\node[circle, draw, fill=black, label={$y_1$}, scale = 0.5] at (A8) {};
			\node[circle, draw, fill=black, label={[xshift=5pt, yshift=-2pt]$x_1$}, scale = 0.5] at (A4) {};
			\node[circle, draw, fill=black, label=right:{$y_2$}, scale = 0.5] at (A2) {};
			\node[circle, draw, fill=black, label=right:{$x_2$}, scale = 0.5] at (A29) {};
			\node[circle, draw, fill=black, label={[xshift=7pt, yshift=-20pt]$y_s$}, scale = 0.5] at (A24) {};
			\node[circle, draw, fill=black, label=below:{$x_s$}, scale = 0.5] at (A21) {};
			\node[circle, draw, fill=black, label=below left:{$y$}, scale = 0.5] at (A18) {};
			\node[circle, draw, fill=black, scale = 0.5] at (A14) {}; 
			\node[circle, draw, fill=black, scale = 0.5] at (A7) {}; 
			
			\node (P') at (-3,0) {$P'$};
			\node (Q0) at (-1.1,2.65) {$Q_0$};
			\node (Q1) at (2.4,1.65) {$Q_1$};
			\node (Qk) at (-1.75,-2.25) {$Q_s$};
			\node (P1) at (1.15,2.55) {$P_{i_1}$};
			\node (P2) at (2.9,0.25) {$P_{i_2}$};
			\node (Pk) at (0,-2.8) {$P_{i_s}$};
		\end{tikzpicture}
		\caption{An illustration of the notation and the decomposition used in the proof of Lemma~\ref{exyeA-decompositions}.
		The (visible part) of $Q$ is red, that of $F$ is green, and matching edges are omitted if favour of clarity.
		\label{fig:exyeA-decompositions}}
	\end{figure}
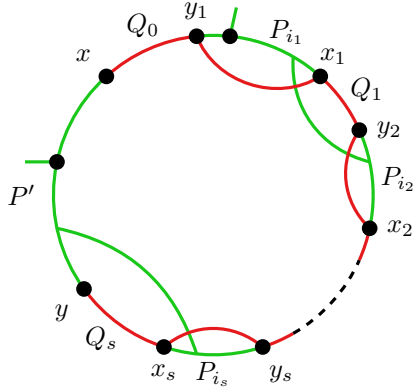
	
	Note that for $Q_i = x_iPy_{i+1}$, where $i\in\{0,\ldots,s\}$ and $x_0=x,y_{s+1}=y$, there are no edges $u'v'$ in $\EXY{\V{\overcirc{Q_i}},\V{\overcirc{Q_j}}}\cap M$, for $i<j$.
	To see this let $m$ be the minimal index in $\{1,\ldots,s\}$ such that $P_m$ is part of $y_{i+1}Px_j$ where $m$ is well-defined as this contains at least the path $y_{i+1}Px_{i+1}$.
	(Be aware that this is not necessarily $P_m$, as we choose the path observed first by the construction, which might appear later in the ordering.)
	Then $u'\in\overcirc{Q_i}, v'\in\overcirc{Q_j}$ are vertices with $P_{m}\subsetneq u'Pv' \subseteq P-\bigcup_{k<m}P_k$, where the last subset relation uses the choice of $m$: 
	Before it, no paths in $x_iPy_{j+1}\supseteq u'Pv'$ were chosen.
	Hence, Property~\ref{prop:maximal-chords} implies that $u'v'\notin M$.
	
	As visualised in \Cref{fig:exyeA-decompositions}, we now set
	\begin{align*}
		Q &= \bigcup_{i=0}^sQ_i + \Set{x_iy_i : i\in\{1,\ldots,s\}} = Q_0y_1x_1Q_1y_2\ldots x_sQ_s,\\
		F'&= P' \cup \bigcup_{i=1}^s P_i +\Set{e_i : \V{P_i}\cap A_2 = \emptyset}, \text{ and}\\
		F &= (\V{C},\E{F'}\cup (E(S,\V{F'})\cap M)) \text{ where } S = \bigcup_{i=0}^s \V{\overcirc{Q_i}}.
	\end{align*}
	Notice that $F'$ is well-defined: 
	The edges $e_i$ exist in the specified case by \ref{prop:A2-or-good-neighbour} and we have just fixed one of them arbitrarily.
	We now show that $F_C=F$, $\calC_C = (\V{C},\E{Q})$, and $M_C = M\setminus(EF\cup EQ)$ is an $(\emptyset,\{x,y\},\emptyset,A_2)$-decomposition of $C$.
	The edges of $G_C$ are partitioned completely as $F\cup Q$ contains all edges of $C$.
	Additionally, $M_C$ is a matching and $\calC_C$ consists of a path from $x$ to $y$ (and isolated vertices).
	We thus need to show that $F$ is a spanning forest in $G_C$ whose components each contain an element of $A_2$, at most one of which is a leaf.
	
	We first show that every vertex of \overcirc{Q_i}, for $i\in\{0,\ldots,s\}$, is either in $A_2$ or adjacent to a vertex of $F'$ in $G[M]$.
	To see this, let $W$ be the set of these vertices and assume $W\neq \V{\overcirc{Q_i}}$.
	Then take two elements $w,z$ of $W\cup\{x_i,y_{i+1}\}$ of minimal distance in $Q_i$ such that $wPz$ contains an element of $\V{\overcirc{Q_i}}\setminus W$.
	As $G-\{w,z\}$ is connected, there exists an edge in $e$ incident to a vertex $u$ of $\overcirc{w}P\overcirc{z}$ with other end in $G-wPz$.
	We show that this edge could be used to extend our sequence, contradicting maximality.
	
	Since $A_2\cup\{x,y\}$ is disjoint from $\overcirc{w}P\overcirc{z}$, the other end $v$ of $e$ is also in $C$.
	By assumption, it is a matching edge that does not end at a vertex in $F'$.		
	But we already know that $e$ cannot end in another $\overcirc{Q_j}$, so it must have both ends in $\overcirc{Q_i}$.		
	Hence it satisfies Conditions~\ref{prop:chord-yield-disjoint-paths} and~\ref{prop:A2-or-good-neighbour}, where the latter holds as either $w$ or $z$ is part of the resulting path and this vertex is in $W$.
	(Note that if $w$ is part of the path, then it cannot be $x_i$ or $y_{i+1}$ as it is an inner vertex of $Q_i$ since both ends of $e$ are and it lies between them.
	The same holds for $z$.)
	Consequently, $e$ either satisfies Property~\ref{prop:maximal-chords} or there exists vertices $u',v'$ as stated there.
	Choosing them to have maximal distance (in $\overcirc{Q_i}$) yields an edge $u'v'$ that satisfies all three conditions.
	In either case we get a contradiction to the maximality of our sequence.
	
	To see that $F$ is a forest, note that $F'$ is the disjoint union of paths with solitary edges connecting those with $\V{P_i}\cap A_2=\emptyset$ to ones prior in the ordering.
	This ensures that $F'$ is acyclic and its leaves are precisely the ends of the paths $P_i$ and $P'$, none of which are in $A_2$.
	The transition to $F$ now only adds edges of $M$ between a vertex of $F'$ and one not in $F'$, which becomes a leaf of the component and is not in $A_2$.
	So $F$ is a spanning forest of $G_C$ without any leaves in $A_2$.
	Now take a component $K$ of $F'$.
	If it contains one of the paths $P_i$ then it has a vertex of $A_2$.
	This just follows from Property~\ref{prop:A2-or-good-neighbour}, for $P_1$ directly and for the others inductively.
	Any component of $F$ that does not contain such a path must be an isolated vertex in $S= \bigcup_{i=0}^s \V{\overcirc{Q_i}}$ without an edge of $M$ to a vertex in $F'$.
	But as vertices of $S$ that do not have such an edge are in $A_2$ by the previous two paragraphs, we get that these components have a vertex of $A_2$, too.
	\qed
\end{proof}

This lemma also requires us to take a look at \cite{XZZ20} again.
Our construction is similar to the one found in Lemma~2.1 there, though our assumption and obtained decomposition differ.
If we formulate their lemma in our notation, we obtain the following.
\begin{lemma}
	Let $x,y\in \partial(C_i)$ and $\partial(C_i)\setminus\{x,y\}\neq\emptyset$.
	The cycle $C_i$ has an $(\emptyset, \{x,y\}, \emptyset, \partial(C_i)\setminus\{x,y\})$-decomposition or there exists an $(A_0,\emptyset, \emptyset, A_2)$-decomposition for any choice of $A_0, A_2$ with $A_0\cup A_2 = \partial(C_i)$, $A_0\cap A_2 = \emptyset$, and $A_2\neq\emptyset$.
\end{lemma}
We do, however, need the first decomposition to exist in our proof and cannot use this to eliminate the 3-connectivity requirement from our lemma.

\section{Proof of the main Theorem}
\label{sec:proof}

To shorten the proof of \Cref{3DC-3-connected-star-like}, we define \emph{good} $I$-decompositions and show that they exist.
These are basically just decompositions of the centre where the required behaviour of the tips corresponds to one of our previous lemmas.
\begin{definition}
	\label{def:good-1-dec}
	Let $I=\{1\}$.
	We call an $I$-decomposition $\calD_I$ of $G$ \emph{good} if every $C_j$ with $j>1$ satisfies that
	\vspace{-5pt}
	\begin{itemize}
		\item $|A_0(\calD_I)\cap \NX{C_j}|$ and $|A_m(\calD_I)\cap \NX{C_j}|$ are at most 1, 
		\item $|A_p(\calD_I)\cap \NX{C_j}|\in\{0,2\}$, and
		\item at most one of these sets is non-empty.
	\end{itemize}
\end{definition}

We begin by showing that good decompositions exist.
\begin{lemma}
	\label{exist-good-1-dec}
	In the situation of Definition~\ref{def:good-1-dec}, $G$ has a good $I$-decomposition.
\end{lemma}
\begin{proof}
	We may assume that $l>1$ as otherwise $G$ is Hamiltonian and has a 3-decomposition.
	We begin by taking a look at the case where $C_1$ has a chord.
	This lets us basically repeat the construction for decompositions given by cycles.
	We choose some minimal cycle $C$ and regard the amount of elements in $\NX{C_j}$ that are on $C$ for cycles $C_j$ with $j>1$.
	If these sets contain at most one element for all cycles, we set $\calC_I =(\V{C_1},\E{C})$ and $T$ is the path in $C_1-\E{C}$.
	This is a tree spanning all degree~3 vertices of $C_1$ except those on $C$.
	As $C$ is minimal, any such vertex is connected to $T$ by an edge of $M$ and we add this to $T$ to obtain $T_I$.
	The remaining edges, which are a subset of $M$, form $M_I$ and this is an $I$-decomposition $\calD_I$.
	For this decomposition the sets $A_m(\calD_I)$ and $A_p(\calD_I)$ are both empty. 
	Furthermore $A_0(\calD_I)\cap\NX{C_j}$ contains at most one element on $C$ and all vertices of $\partial(C_1)$ that are not in $C$ are in $A_2(\calD_I)$, so the conditions of a good decomposition are satisfied.
	This case is illustrated in \Cref{fig:good-1-dec-no-multiple-edges}.
	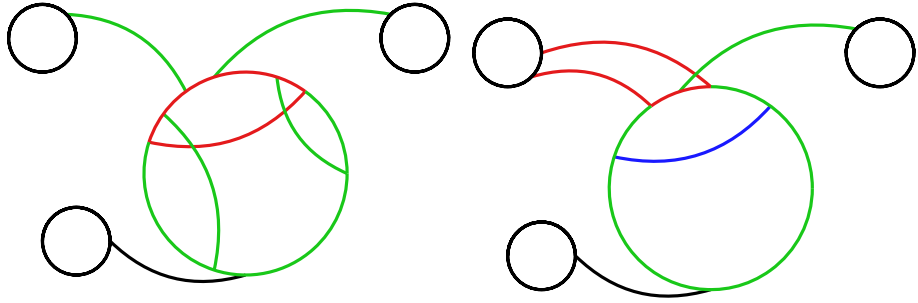
\begin{figure}[htb]
		\centering
		\begin{subfigure}{.49\textwidth}
			\centering
				\begin{tikzpicture}[scale=0.9]
				\placeCycleCoordinates{0,0}{1.5}{20}{A}{}
				\placeCycleCoordinates{-3,2}{0.5}{8}{B}{}
				\placeCycleCoordinates{2.5,2}{0.5}{8}{C}{}
				\placeCycleCoordinates{-2.5,-1}{0.5}{8}{D}{}
				\draw[color = ccol] (A3.center) to[bend left=30] (A9.center);
				\draw[color = tcol] (A4.center) to[bend right=30] (A20.center);
				\draw[color = tcol] (A8.center) to[bend left=30] (A14.center);
				\draw[color = tcol] (A7.center) to[bend right=30] (B1.center);
				\draw[color = tcol] (A6.center) to[bend left=30] (C3.center);
				\draw[color = black] (D8.center) to[bend right=30] (A15.center);
				\drawCycle{0,0}{1.5}{20}{}{{0/3/tcol,3/9/ccol,9/20/tcol}}	
				\drawCycle{-3,2}{0.5}{8}{}{{0/20/black}}
				\drawCycle{2.5,2}{0.5}{8}{}{{0/20/black}}			
				\drawCycle{-2.5,-1}{0.5}{8}{}{{0/20/black}}							
			\end{tikzpicture}
			\caption{$T_I$ in case no cycle has multiple edges to $C$.
			\label{fig:good-1-dec-no-multiple-edges}}
		\end{subfigure}
		\hfill
		\begin{subfigure}{.49\textwidth}
			\centering
			\begin{tikzpicture}[scale=0.9]
				\placeCycleCoordinates{0,0}{1.5}{20}{A}{}
				\placeCycleCoordinates{-3,2}{0.5}{8}{B}{}
				\placeCycleCoordinates{2.5,2}{0.5}{8}{C}{}
				\placeCycleCoordinates{-2.5,-1}{0.5}{8}{D}{}
				\draw[color = mcol] (A3.center) to[bend left=30] (A9.center);
				\draw[color = ccol] (A7.center) to[bend right=30] (B7.center);
				\draw[color = ccol] (A5.center) to[bend right=30] (B8.center);
				\draw[color = tcol] (A6.center) to[bend left=30] (C3.center);
				\draw[color = black] (D8.center) to[bend right=30] (A15.center);
				\drawCycle{0,0}{1.5}{20}{}{{0/5/tcol,5/7/ccol,7/20/tcol}}	
				\drawCycle{-3,2}{0.5}{8}{}{{0/20/black}}
				\drawCycle{2.5,2}{0.5}{8}{}{{0/20/black}}			
				\drawCycle{-2.5,-1}{0.5}{8}{}{{0/20/black}}	
			\end{tikzpicture}
			\caption{And $T_I$ in case the cycle $C_k$ does.
			\label{fig:good-1-dec-multiple-edges}}
		\end{subfigure}
		\caption{The decompositions used in the proof of Lemma~\ref{exist-good-1-dec} when the centre has a chord.}
	\end{figure}
	
	Next we assume that $C$ contains multiple vertices of some $\NX{C_j}$, then we choose two from the same such set $\NX{C_k}$ of minimal distance and denote the path between them in $P_C$ by $P$.
	Then this path contains at most one vertex from sets $\NX{C_j}$ for $j\neq k, j>1$.
	We apply an analogous construction, where we set $\calC_I = (\V{C_I},\E{P})$ and $T$ to the path $C_i-\overcirc{P}$.
	Again we connect vertices of degree~3 in $G_I$ that are not part of $T$ yet and obtain $T_I$ and an $I$-decomposition.
	This, too, is good as $A_m(\calD_I)$ is still empty, $A_p(\calD_I)$ contains exactly the two ends of $P$, which are in \NX{C_k}, and $A_0(\calD_I)\cap\NX{C_j}$ is empty for $j=k$ and has at most one element otherwise.
	\Cref{fig:good-1-dec-multiple-edges} shows the decomposition constructed in this case.
	
	Finally, in the case that $C_1$ is chordless, we take two adjacent vertices $u,v$ on $C_1$.
	If they have neighbours in different cycles, we set $T_I=C_i-uv$ and $M_I=\{uv\}$, leaving $\calC_I$ empty.
	This gives us a spanning tree and a matching that form a good $I$-decomposition $\calD_I$ as all vertices are in $A_2(\calD_I)$ except for $u,v$ which are part of $A_m(\calD_I)$ and in different sets $\NX{C_j}$.
	Should $u,v$ be in the same set $\NX{C_j}$, then we add $uv$ to $\calC_I$ instead and obtain another good $I$-decomposition, this time with $A_p(\calD_I)=\{u,v\}$ and these being part of the same set $\NX{C_j}$.
	\qed
\end{proof}

Now we can finally finish up the proof of \Cref{3DC-3-connected-star-like}.
\begin{proof}
	Let $I=\{1\}$.
	By Lemma~\ref{exist-good-1-dec} there exists a good $I$-decomposition of $G$.
	We now iteratively extend this decomposition to more cycles by checking the conditions of Lemma~\ref{expanding-I-decompositions} and verifying that we can satisfy them with the help of Corollary~\ref{eeeA-decompositions} and Lemmas~\ref{eexA-decompositions}, \ref{xeeA-decompositions}, and \ref{exyeA-decompositions}.
	
	As long as $I\neq \{1,\ldots,l\}$ let us take an element $i\notin I$ and set $J=I\cup\{i\}$.
	Then we can apply Lemma~\ref{expanding-I-decompositions} which gives us a $J$-decomposition if we can exhibit an $(A_0,A_p,A_m,A_2)$-decomposition where $A_x = \NX{A_x(\calD_I)}\cap \V{C_i}$, $x\in\{0,p,m,2\}$.
	As $G$ is star-like, we know that $\partial(G_I)\cap \NX{C_i} = \partial(G_1)\cap \NX{C_i}$.
	Using that $\calD_{\{1\}}$ is good we can conclude that all vertices in $\partial(G_I)\cap \NX{C_i}$ are in $A_2(\calD_I)$, with the possible exception of either one element in $A_0(\calD_I)$, one in $A_m(\calD_I)$, or two in $A_p(\calD_I)$.
	Consequently, we have that all vertices in $\partial(C_i)$ are in $A_2$ aside from a single one in either $A_0$ or $A_m$, or two in $A_p$.
	(Note that this is true initially and remains true in later steps as Lemma~\ref{expanding-I-decompositions} ensures that edges in the centre are never reassigned.)
	The set $\partial(C_i)$ contains at least three elements as it separates $C_i$ from $C_1$ in $G$ and $G$ is 3-connected.
	Hence, $A_2\neq\emptyset$.
	We have thus fulfilled the premise of Corollary~\ref{eeeA-decompositions} and Lemmas~\ref{eexA-decompositions}, \ref{xeeA-decompositions}, and \ref{exyeA-decompositions}, giving us an $(A_0,A_p,A_m,A_2)$-decomposition and completing the proof.
	\qed
\end{proof}

\section{Constructing 3-connected star-like Graphs for which the Conjecture was not already known}
\label{sec:example}
In this section we construct 3-connected star-like graphs, for which \Cref{3DC-3-connected-star-like} shows that they have a 3-de\-com\-po\-si\-tion.
As the conjecture is already proved for graphs that are traceable, planar, claw-free, 3-connected and of tree-width at most 3, embeddable in the Torus or Klein bottle, or have a matching with a contraction graph of order at most~3, our goal is to find examples that have none of these properties.
The construction we present is closely based on \cite{FS07}, which we have modified in order to obtain graphs that are actually star-like.
We recall some notions and results from this paper that we need.
\begin{definition}
	A graph $H$ is hypohamiltonian if it is not Hamiltonian, but $H-v$ is for all $v\in V$.
\end{definition}
\begin{observation}
	\label{hypohamiltonian-ham-paths}
	For a hypohamiltonian graph $H$ and $z\in H$, $H-z$ has no Hamiltonian path between two neighbours of $z$, as this could be extended to a Hamiltonian cycle of $H$.
\end{observation}

For the actual construction, let $H_i$, $i\in I=\{1,2,3\}$, be graphs that are 3-connected, cubic, non-planar, and hypohamiltonian.
Also let $G_4=(\{x\},\emptyset)$ be a further vertex.
In order to see that such $H_i$ exist and to obtain infinitely many examples, we exhibit an infinite family of candidates for the $H_i$.
First, note that we can drop the 3-connectivity requirement.
\begin{observation}
	Any hypohamiltonian graph is 3-connected.
\end{observation}
\begin{proof}
	Let $G$ by a hypohamiltonian graph and $v\in \V{G}$.
	Since $G-v$ is Hamiltonian, it has at least three vertices and $G$ has order at least~4.
	Next, let $u$ be another vertex.
	Then, since $G-v$ has a Hamiltonian cycle, $G-\{u,v\}$ is connected.
\end{proof}
So we only need to find a family of cubic graphs that are hypohamiltonian and non-planar.
The Petersen graph and the family of flower snarks satisfies both these properties.
We start with the former.
\begin{lemma}
	The Petersen graph is cubic and hypohamiltonian.
	It also has a $K_5$ as a minor and is thus non-planar.
\end{lemma}
\begin{proof}
	All these properties can be found in \cite{HS93}.
\end{proof}

\begin{lemma}
	The flower snarks $J_k$ for all odd $k\geq 5$ are cubic, non-planar, and hypohamiltonian.
\end{lemma}
\begin{proof}
	The flower snarks are cubic by definition and hypohamiltonian by~\cite{CE83}.
	A proof that these graphs are non-planar can be found in~\cite{Isa75}.
\end{proof}

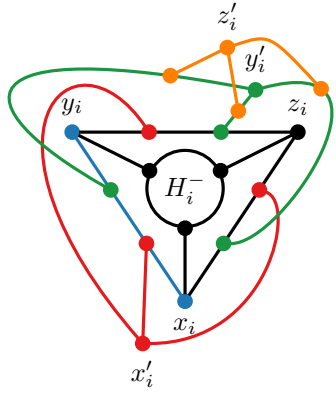
\begin{figure}[hbt]
	\centering
	\begin{tikzpicture}[scale=0.75]
		\node[circle, draw] (h) at (0,0) {$H_i^-$};
		\node[circle,draw=stdblue,fill=stdblue,scale=0.5,label=above:{$y_i$}] (y) at (-2,1) {};
		\node[circle,draw,fill=black,scale=0.5,label=above:{$z_i$}] (z) at (2,1) {};
		\node[circle,draw=stdblue,fill=stdblue,scale=0.5,label=below:{$x_i$}] (x) at (0,-2) {};
		\node[circle,draw=stdgreen,fill=stdgreen,scale=0.5,label=above:{$y'_i$}] (y2) at (1.25,1.75) {};
		\node[circle,draw=orange,fill=orange,scale=0.5,label=above:{$z'_i$}] (z2) at (0.75,2.5) {};
		\node[circle,draw=stdred,fill=stdred,scale=0.5,label=below:{$x'_i$}] (x2) at (-0.75,-2.75) {};
		\draw (x) to node[pos=1,circle,draw,fill=black,scale=0.5] {} (h);
		\draw (y) to node[pos=1,circle,draw,fill=black,scale=0.5] {} (h) ;
		\draw (z) to node[pos=1,circle,draw,fill=black,scale=0.5] {} (h) ;
		\draw[stdblue] (x) to node[pos=0.33,circle,draw=stdred,fill=stdred,scale=0.5] (c1) {} 
		node[pos=0.67,circle,draw=stdgreen,fill=stdgreen,scale=0.5] (d1) {} (y);
		\draw (y) to node[pos=0.33,circle,draw=stdred,fill=stdred,scale=0.5] (c2) {} 
		node[pos=0.67,circle,draw=stdgreen,fill=stdgreen,scale=0.5] (d2) {} (z);
		\draw (z) to node[pos=0.33,circle,draw=stdred,fill=stdred,scale=0.5] (c3) {} 
		node[pos=0.67,circle,draw=stdgreen,fill=stdgreen,scale=0.5] (d3) {} (x);
		\draw[stdred] (x2) to (c1);
		\draw[stdred] (x2) .. controls (-4,1) and (-2,3) .. (c2);
		\draw[stdred] (x2) .. controls (1.5,-2.75) and (2,-0.25) .. (c3);      
		\draw[stdgreen] (y2) to node[pos=0.5,circle,draw=orange,fill=orange,scale=0.5] (a) {} (d2);
		\draw[stdgreen] (y2) .. controls (4,2.5) and (1.75,-1) .. (d3) node[pos=0.2,circle,draw=orange,fill=orange,scale=0.5] (b) {};
		\draw[stdgreen] (y2) .. controls (-4,2.75) and (-4,1.5) .. (d1) node[pos=0.1,circle,draw=orange,fill=orange,scale=0.5] (c) {}; 
		\draw[orange] (a) to (z2);
		\draw[orange] (b) .. controls (1.5,2.75) .. (z2);
		\draw[orange] (c) to (z2);        
	\end{tikzpicture}
	\caption{The extension $G_i$ of the graph $H_i$.
	\label{fig:example-graph-extending-h}}
\end{figure}
Now on to the construction.
For each $i\in I$ we pick some vertex $z_i\in H_i$ and set $H_i^-$ to $H_i-z_i$.
Our next goal is to expand $H_i$ to the slightly larger 3-connected cubic graph $G_i$ shown in \Cref{fig:example-graph-extending-h}.
To ensure 3-connectivity we may iteratively subdivide two distinct edges and connect the resulting degree two vertices \cite{Wor79}.
We apply this first to two of the edges incident to $z_i$ and call the subdivision vertices $x_i$ and $y_i$.
This step is drawn in blue in the figure.
Next note that we can also subdivide three edges and connect the subdivision vertices to a single new vertex since this is just a sequence of two subdivision steps.
We apply this to the edges of the triangle $x_iy_iz_ix_i$ and call the new vertex $x_i'$, which is drawn in red.
Next, we subdivide three of the newly split edges of the triangle that form a matching and connect them to the new vertex $y_i'$, shown in green.
Finally, we subdivide all edges incident to $y_i'$ and connect them to $z_i'$, which are the orange vertices and edges.
Let the resulting graph be called $G_i$, denote $G_i-z_i'$ by $G_i^-$, and set $\NX[G_i]{z_i'} = \{a_i,b_i,c_i\}$.

Before we go on, we show some properties that the graphs $G_i$ and $G_i^-$ possess.
\begin{lemma}
	\label{prop-of-gi-minus}
	The following properties hold:
	\begin{enumrom}
		\item The graph $G_i$ is 3-connected and cubic. \label{prop-of-gi-minus-3-conn}
		\item The graph $G_i^-$ has $H_i$ as a minor. \label{prop-of-gi-minus-non-planar}
		\item The graph $G_i^-$ has no Hamiltonian path with both ends in $\NX[G_i]{z_i'}$. \label{prop-of-gi-minus-Ham-paths}
		\item For any $u\in G_i^-$, there exist three $u\NX[G_i]{z_i'}$-paths that are disjoint aside from $u$. \label{prop-of-gi-minus-exit-paths}
		\item For any pair of distinct vertices $u,v\in \NX[G_i]{z_i'}$, $G_i^--H_i^-$ contains a Hamiltonian $uv$-path $P_i(u,v)$.\label{prop-of-gi-minus-path-around-hi}
		\item The graph $H_i^-$ is Hamiltonian.\label{prop-of-gi-minus-Ham}
	\end{enumrom}
\end{lemma}	
\begin{proof}
	We look at the six parts in turn.
	\begin{enumrom}
		\item This follows by the assumption on $H_i$ and the construction.
		\item The minor $H_i$ is obtained by taking the subgraph consisting only of the black edges in \Cref{fig:example-graph-extending-h} and removing the subdivision vertices, which corresponds to contracting one of their incident edges.
		\item Notice that a Hamiltonian path in $G_i^-$ that does not end in $H_i^-$ must use exactly two of the three edges between $\{x_i,y_i,z_i\}$ and $\NX[H_i]{z_i}$ (as they form a cut).
		Thus it would induce a Hamiltonian path in $H_i^-$ with both ends in $\NX[H_i]{z_i}$, which does not exist by Observation~\ref{hypohamiltonian-ham-paths}.
		\item As $G_i$ is 3-connected, it contains three internally vertex-disjoint $uz_i'$-paths, which yield the desired $u\NX[G_i]{z_i'}$-paths in $G_i^-$.
		\item The three desired paths are shown in \Cref{fig:example-graph-hamiltonian-paths-in-G-H}.
		Note that the graph is symmetric and the second two are all just rotations of the first.
		\item Follows directly from $H_i$ being hypohamiltonian.\qed
	\end{enumrom}
\end{proof}
\begin{figure}[htb]
	\centering
	\scalebox{0.8}[0.8]{
	\begin{tikzpicture}[scale=0.55]
		\node[circle,draw,fill=black,scale=0.5] (y) at (-2,1) {};
		\node[circle,draw,fill=black,scale=0.5] (z) at (2,1) {};
		\node[circle,draw,fill=black,scale=0.5] (x) at (0,-2) {};
		\node[circle,draw,fill=black,scale=0.5] (y2) at (1.25,1.75) {};
		\node[circle,draw,fill=black,scale=0.5] (x2) at (-0.75,-2.75) {};
		\draw (x.center) to node[pos=0.33,circle,draw,fill=black,scale=0.5] (c1) {} 
		node[pos=0.67,circle,draw,fill=black,scale=0.5] (d1) {} (y.center);
		\draw (y.center) to node[pos=0.33,circle,draw,fill=black,scale=0.5] (c2) {} 
		node[pos=0.67,circle,draw,fill=black,scale=0.5] (d2) {} (z.center);
		\draw (z.center) to node[pos=0.33,circle,draw,fill=black,scale=0.5] (c3) {} 
		node[pos=0.67,circle,draw,fill=black,scale=0.5] (d3) {} (x.center);
		\draw (x2.center) to (c1.center);
		\draw (x2.center) .. controls (-4,1) and (-2,3) .. (c2.center);
		\draw (x2.center) .. controls (1.5,-2.75) and (2,-0.25) .. (c3.center);      
		\draw (y2.center) to node[pos=0.5,circle,draw,fill=black,scale=0.5] (a) {} (d2.center);
		\draw[] (y2.center) .. controls (4,2.5) and (1.75,-1) .. (d3.center) node[pos=0.2,circle,draw,fill=black,scale=0.5] (b) {};
		\draw (y2.center) .. controls (-4,2.75) and (-4,1.5) .. (d1.center) node[pos=0.1,circle,draw,fill=black,scale=0.5] (c) {}; 
		
		\draw[draw,line width=4pt,stdred,opacity=0.8,join=round] (a.center) to (d2.center) to (z.center) to (x.center) to (c1.center) to (x2.center) .. controls (-4,1) and (-2,3) .. (c2.center) to (y.center) to (d1.center);
		\draw[draw,line width=4pt,stdred,opacity=0.8,join=round] (y2.center) .. controls (-4,2.75) and (-4,1.5) .. (d1.center);
		\draw[draw,line width=4pt,red,opacity=0.8,join=round,dash pattern=on 17pt off 1000pt] (y2.center) .. controls (4,2.5) and (1.75,-1) .. (d3.center);
		
		\node[circle,draw,fill=black,scale=0.5] (y) at (-2,1) {};
		\node[circle,draw,fill=black,scale=0.5] (z) at (2,1) {};
		\node[circle,draw,fill=black,scale=0.5] (x) at (0,-2) {};
		\node[circle,draw,fill=black,scale=0.5] (y2) at (1.25,1.75) {};
		\node[circle,draw,fill=black,scale=0.5] (x2) at (-0.75,-2.75) {};
		\draw (x.center) to node[pos=0.33,circle,draw,fill=black,scale=0.5] (c1) {} 
		node[pos=0.67,circle,draw,fill=black,scale=0.5] (d1) {} (y.center);
		\draw (y.center) to node[pos=0.33,circle,draw,fill=black,scale=0.5] (c2) {} 
		node[pos=0.67,circle,draw,fill=black,scale=0.5] (d2) {} (z.center);
		\draw (z.center) to node[pos=0.33,circle,draw,fill=black,scale=0.5] (c3) {} 
		node[pos=0.67,circle,draw,fill=black,scale=0.5] (d3) {} (x.center);
		\draw (x2.center) to (c1.center);
		\draw (x2.center) .. controls (-4,1) and (-2,3) .. (c2.center);
		\draw (x2.center) .. controls (1.5,-2.75) and (2,-0.25) .. (c3.center);      
		\draw (y2.center) to node[pos=0.5,circle,draw,fill=black,scale=0.5] (a) {} (d2.center);
		\draw[] (y2.center) .. controls (4,2.5) and (1.75,-1) .. (d3.center) node[pos=0.2,circle,draw,fill=black,scale=0.5] (b) {};
		\draw (y2.center) .. controls (-4,2.75) and (-4,1.5) .. (d1.center) node[pos=0.1,circle,draw,fill=black,scale=0.5] (c) {}; 
	\end{tikzpicture}
	}
	\hfill
	\scalebox{0.8}[0.8]{
	\begin{tikzpicture}[scale=0.55]
		\node[circle,draw,fill=black,scale=0.5] (y) at (-2,1) {};
		\node[circle,draw,fill=black,scale=0.5] (z) at (2,1) {};
		\node[circle,draw,fill=black,scale=0.5] (x) at (0,-2) {};
		\node[circle,draw,fill=black,scale=0.5] (y2) at (1.25,1.75) {};
		\node[circle,draw,fill=black,scale=0.5] (x2) at (-0.75,-2.75) {};
		\draw (x.center) to node[pos=0.33,circle,draw,fill=black,scale=0.5] (c1) {} 
		node[pos=0.67,circle,draw,fill=black,scale=0.5] (d1) {} (y.center);
		\draw (y.center) to node[pos=0.33,circle,draw,fill=black,scale=0.5] (c2) {} 
		node[pos=0.67,circle,draw,fill=black,scale=0.5] (d2) {} (z.center);
		\draw (z.center) to node[pos=0.33,circle,draw,fill=black,scale=0.5] (c3) {} 
		node[pos=0.67,circle,draw,fill=black,scale=0.5] (d3) {} (x.center);
		\draw (x2.center) to (c1.center);
		\draw (x2.center) .. controls (-4,1) and (-2,3) .. (c2.center);
		\draw (x2.center) .. controls (1.5,-2.75) and (2,-0.25) .. (c3.center);      
		\draw (y2.center) to node[pos=0.5,circle,draw,fill=black,scale=0.5] (a) {} (d2.center);
		\draw (y2.center) .. controls (4,2.5) and (1.75,-1) .. (d3.center) node[pos=0.2,circle,draw,fill=black,scale=0.5] (b) {};
		\draw (y2.center) .. controls (-4,2.75) and (-4,1.5) .. (d1.center) node[pos=0.1,circle,draw,fill=black,scale=0.5] (c) {}; 
		
		\draw[draw,line width=4pt,stdred,opacity=0.8,join=round,dash pattern=on 0pt off 18pt on 1000pt off 0pt] (y2.center) .. controls (4,2.5) and (1.75,-1) .. (d3.center);
		\draw[draw,line width=4pt,stdred,opacity=0.8,join=round] (d3.center) to (x.center) to (y.center) to (c2.center);
		\draw[draw,line width=4pt,stdred,opacity=0.8,join=round] (x2.center) .. controls (-4,1) and (-2,3) .. (c2.center);
		\draw[draw,line width=4pt,stdred,opacity=0.8,join=round] (x2.center) .. controls (1.5,-2.75) and (2,-0.25) .. (c3.center) to (z.center) to (d2.center) to (y2.center);	
		\draw[draw,line width=4pt,stdred,opacity=0.8,join=round,dash pattern=on 23pt off 1000pt] (y2.center) .. controls (-4,2.75) and (-4,1.5) .. (d1.center);
		
		\node[circle,draw,fill=black,scale=0.5] (y) at (-2,1) {};
		\node[circle,draw,fill=black,scale=0.5] (z) at (2,1) {};
		\node[circle,draw,fill=black,scale=0.5] (x) at (0,-2) {};
		\node[circle,draw,fill=black,scale=0.5] (y2) at (1.25,1.75) {};
		\node[circle,draw,fill=black,scale=0.5] (x2) at (-0.75,-2.75) {};
		\draw (x.center) to node[pos=0.33,circle,draw,fill=black,scale=0.5] (c1) {} 
		node[pos=0.67,circle,draw,fill=black,scale=0.5] (d1) {} (y.center);
		\draw (y.center) to node[pos=0.33,circle,draw,fill=black,scale=0.5] (c2) {} 
		node[pos=0.67,circle,draw,fill=black,scale=0.5] (d2) {} (z.center);
		\draw (z.center) to node[pos=0.33,circle,draw,fill=black,scale=0.5] (c3) {} 
		node[pos=0.67,circle,draw,fill=black,scale=0.5] (d3) {} (x.center);
		\draw (x2.center) to (c1.center);
		\draw (x2.center) .. controls (-4,1) and (-2,3) .. (c2.center);
		\draw (x2.center) .. controls (1.5,-2.75) and (2,-0.25) .. (c3.center);      
		\draw (y2.center) to node[pos=0.5,circle,draw,fill=black,scale=0.5] (a) {} (d2.center);
		\draw (y2.center) .. controls (4,2.5) and (1.75,-1) .. (d3.center) node[pos=0.2,circle,draw,fill=black,scale=0.5] (b) {};
		\draw (y2.center) .. controls (-4,2.75) and (-4,1.5) .. (d1.center) node[pos=0.1,circle,draw,fill=black,scale=0.5] (c) {}; 
	\end{tikzpicture}
	}
	\hfill
	\scalebox{0.8}[0.8]{
	\begin{tikzpicture}[scale=0.55]
		\node[circle,draw,fill=black,scale=0.5] (y) at (-2,1) {};
		\node[circle,draw,fill=black,scale=0.5] (z) at (2,1) {};
		\node[circle,draw,fill=black,scale=0.5] (x) at (0,-2) {};
		\node[circle,draw,fill=black,scale=0.5] (y2) at (1.25,1.75) {};
		\node[circle,draw,fill=black,scale=0.5] (x2) at (-0.75,-2.75) {};
		\draw (x.center) to node[pos=0.33,circle,draw,fill=black,scale=0.5] (c1) {} 
		node[pos=0.67,circle,draw,fill=black,scale=0.5] (d1) {} (y.center);
		\draw (y.center) to node[pos=0.33,circle,draw,fill=black,scale=0.5] (c2) {} 
		node[pos=0.67,circle,draw,fill=black,scale=0.5] (d2) {} (z.center);
		\draw (z.center) to node[pos=0.33,circle,draw,fill=black,scale=0.5] (c3) {} 
		node[pos=0.67,circle,draw,fill=black,scale=0.5] (d3) {} (x.center);
		\draw (x2.center) to (c1.center);
		\draw (x2.center) .. controls (-4,1) and (-2,3) .. (c2.center);
		\draw (x2.center) .. controls (1.5,-2.75) and (2,-0.25) .. (c3.center);      
		\draw (y2.center) to node[pos=0.5,circle,draw,fill=black,scale=0.5] (a) {} (d2.center);
		\draw (y2.center) .. controls (4,2.5) and (1.75,-1) .. (d3.center) node[pos=0.2,circle,draw,fill=black,scale=0.5] (b) {};
		\draw (y2.center) .. controls (-4,2.75) and (-4,1.5) .. (d1.center) node[pos=0.1,circle,draw,fill=black,scale=0.5] (c) {};    
		
		\draw[draw,line width=4pt,stdred,opacity=0.8,join=round,dash pattern=on 0pt off 25pt on 1000pt off 0pt] (y2.center) .. controls (-4,2.75) and (-4,1.5) .. (d1.center);
		\draw[draw,line width=4pt,stdred,opacity=0.8,join=round] (d1.center) to (y.center) to (z.center) to (c3.center);
		\draw[draw,line width=4pt,stdred,opacity=0.8,join=round] (x2.center) .. controls (1.5,-2.75) and (2,-0.25) .. (c3.center);
		\draw[draw,line width=4pt,stdred,opacity=0.8,join=round] (x2.center) to (c1.center) to (x.center) to (d3.center);
		\draw[draw,line width=4pt,stdred,opacity=0.8,join=round] (y2.center) .. controls (4,2.5) and (1.75,-1) .. (d3.center);
		\draw[draw,line width=4pt,stdred,opacity=0.8,join=round] (y2.center) to (a.center);
		
		\node[circle,draw,fill=black,scale=0.5] (y) at (-2,1) {};
		\node[circle,draw,fill=black,scale=0.5] (z) at (2,1) {};
		\node[circle,draw,fill=black,scale=0.5] (x) at (0,-2) {};
		\node[circle,draw,fill=black,scale=0.5] (y2) at (1.25,1.75) {};
		\node[circle,draw,fill=black,scale=0.5] (x2) at (-0.75,-2.75) {};
		\draw (x.center) to node[pos=0.33,circle,draw,fill=black,scale=0.5] (c1) {} 
		node[pos=0.67,circle,draw,fill=black,scale=0.5] (d1) {} (y.center);
		\draw (y.center) to node[pos=0.33,circle,draw,fill=black,scale=0.5] (c2) {} 
		node[pos=0.67,circle,draw,fill=black,scale=0.5] (d2) {} (z.center);
		\draw (z.center) to node[pos=0.33,circle,draw,fill=black,scale=0.5] (c3) {} 
		node[pos=0.67,circle,draw,fill=black,scale=0.5] (d3) {} (x.center);
		\draw (x2.center) to (c1.center);
		\draw (x2.center) .. controls (-4,1) and (-2,3) .. (c2.center);
		\draw (x2.center) .. controls (1.5,-2.75) and (2,-0.25) .. (c3.center);      
		\draw (y2.center) to node[pos=0.5,circle,draw,fill=black,scale=0.5] (a) {} (d2.center);
		\draw (y2.center) .. controls (4,2.5) and (1.75,-1) .. (d3.center) node[pos=0.2,circle,draw,fill=black,scale=0.5] (b) {};
		\draw (y2.center) .. controls (-4,2.75) and (-4,1.5) .. (d1.center) node[pos=0.1,circle,draw,fill=black,scale=0.5] (c) {};     
	\end{tikzpicture}  
	}
	\caption{The three Hamiltonian paths between pairs of neighbours of $z_i$ in $G_i^--H_i^-$.
	\label{fig:example-graph-hamiltonian-paths-in-G-H}}
\end{figure}
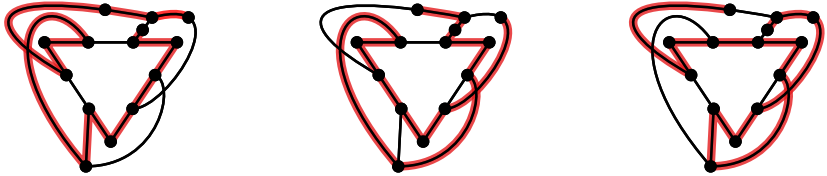
With this out of the way we now construct our desired graph $G=K_4[G_1,G_2,G_3]$ as follows.
Take $G_1^-\cup G_2^-\cup G_3^-\cup G_4$ and add the following six edges: $xa_1$, $xa_2$, $xa_3$, $b_1c_2$, $b_2c_3$, and $b_3c_1$.
This graph is visualised in \Cref{fig:example-graph-k4g1g2g3}.
\begin{figure}[htb]
	\centering
	\begin{tikzpicture}[scale=0.95]
		\node[circle, draw] (g1) at (0,2) {$G_1^-$};
		\node[circle, draw] (g2) at (-1.75,-1) {$G_2^-$};
		\node[circle, draw] (g3) at (1.75,-1) {$G_3^-$};
		\node[circle,draw,fill=black,scale=0.5,label=above right:{$x$}] (x) at (0,0) {};
		
		\draw (g1) to node[pos=0,circle,draw,fill=black,scale=0.5,label=below left:{$a_1$}] {} (x);
		\draw (g2) to node[pos=0,circle,draw,fill=black,scale=0.5,label=right:{$a_2$}] {} (x);
		\draw (g3) to node[pos=0,circle,draw,fill=black,scale=0.5,label=left:{$a_3$}] {} (x);
		\draw (g1.west) to node[pos=0,circle,draw,fill=black,scale=0.5,label=left:{$b_1$}] {} 
		node[pos=1,circle,draw,fill=black,scale=0.5,label=above left:{$c_2$}] {} (g2.north);
		\draw (g2.south east) to node[pos=0,circle,draw,fill=black,scale=0.5,label=below:{$b_2$}] {} 
		node[pos=1,circle,draw,fill=black,scale=0.5,label=below:{$c_3$}] {} (g3.south west);
		\draw (g3.north) to node[pos=0,circle,draw,fill=black,scale=0.5,label=above right:{$b_3$}] {} 
		node[pos=1,circle,draw,fill=black,scale=0.5,label=right:{$c_1$}] {} (g1.east);
	\end{tikzpicture}
	\caption{The graph $G = K_4[G_1,G_2,G_3]$ that serves as our example.
	\label{fig:example-graph-k4g1g2g3}}
\end{figure}
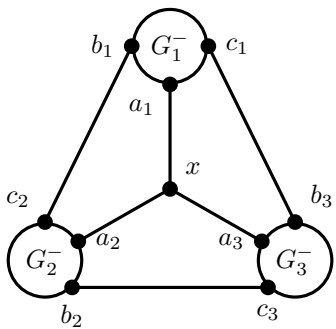

We now prove that this does the trick.
\begin{theorem}
	The graph $G$ is cubic, non-traceable, star-like, and 3-connected.
\end{theorem}
\begin{proof}
	The graph is cubic by construction. 
	To see that $G$ is not traceable, we simply need to realise that a Hamiltonian path $P$ would yield a graph $G_i^-$ in which no end of the path resides.
	This means that $P$ restricted to $G_i^-$ would necessarily be a Hamiltonian path with ends in \NX[G_i]{z_i'}, which does not exits by Lemma~\ref{prop-of-gi-minus}~\ref{prop-of-gi-minus-Ham-paths}.
	
	In order to prove that $G$ is star-like, we just need to specify the cycles.
	We take a Hamiltonian cycle in each $H_i^-$ for $i\in I$, which exist by Lemma~\ref{prop-of-gi-minus}~\ref{prop-of-gi-minus-Ham}.
	In addition, we use the paths given by Lemma~\ref{prop-of-gi-minus}~\ref{prop-of-gi-minus-path-around-hi} to obtain a final Hamiltonian cycle $C$ in $G-\bigcup_{i\in I}H_i^-$, namely
	\begin{displaymath}
		C = xa_1P_1(a_1,b_1)b_1c_2P_2(c_2,b_2)b_2c_3P_3(c_3,a_3)a_3x.
	\end{displaymath}
	This yields a decomposition where the contraction graph is a star in which the centre corresponds to $C$ and the three tips to the cycles in the $H_i^-$.
	
	Finally, we are only left with the proof that $G$ is 3-connected.
	To this end we prove the existence of three internally vertex-disjoint paths between any pair $u,v$ of distinct vertices in $G$.
	Let $u\in G_1^-$ without loss of generality.
	If $v\in G_1^-$ as well, then $G/(VG\setminus VG_1^-) \cong G_1$ contains three $uv$-paths by the 3-connectivity of $G_1$.
	If one such path uses the super node, we replace it by a path through the subgraph we contracted.
	So assume that $v\notin G_1^-$.
	If $v\in G_2^-$ ($G_3^-$ is analogous), we use Lemma~\ref{prop-of-gi-minus}~\ref{prop-of-gi-minus-exit-paths} to reduce the problem to finding three disjoint $\NX[G_1]{z_1'}\NX[G_2]{z_2'}$-paths.
	But these exist, just take
	\begin{displaymath}
		P_1 = a_1xa_2, \quad P_2 = b_1c_2, \quad P_3 = c_1b_3P_3(b_3,c_3)c_3b_2.
	\end{displaymath}
	Similarly we deal with the case that $v=x$, where it suffices to find three $x\NX{z_1'}$-paths that are disjoint aside from $x$.
	This time we use
	\begin{displaymath}
		P_1 = xa_1, \quad P_2 = xP_2a_2(a_2,c_2)c_2b_1, \quad P_3 = xa_3P_3(a_3,b_3)b_3c_1.
	\end{displaymath}
	This completes the proof.\qed
\end{proof}

\begin{theorem}
	The graph $G$ has genus and non-orientable genus at least~3.
\end{theorem}
\begin{proof}
	As the $H_i$ are non-planar, they each have a $K_5$- or $K_{3,3}$-minor $H_i'$ for $i\in\{1,2,3\}$.
	By Lemma~\ref{prop-of-gi-minus}~\ref{prop-of-gi-minus-non-planar}, the graph $G_i^-$ contains an $H_i'$-minor.
	As the $G_i^-$ are connected, we may assume that all their nodes are contained in some super node of $H_i'$.
	By taking these minors and removing the edges $b_1c_2,b_2c_3,b_3c_1$ in $G$, we obtain the minor $H_1'\cup H_2'\cup H_3'\cup G_4+E'$ where $E'$ contains the three edges from $x$ to the super nodes containing the vertices $a_1,a_2,a_3$.
	By contracting the edges in $E'$ as well, we obtain a connected graph $H$ with the three blocks $H_1',H_2',H_3'$.
	
	The genus of the $H_i'$ is~1 and so is the non-orientable genus, as they are non-planar but can be embedded in the Torus or the projective plane, respectively.
	Consequently, the genus of $H$ is~3, because it is the sum of the genuses of its blocks by \cite{BHKY62}.
	We also have that $H$ is not orientably simple by \cite[Theorem~1]{BS77} as $K_5$ and $K_{3,3}$ are not.
	Thus the non-orientable genus can be computed as specified in Corollary~3 of the same paper, giving us 
	\begin{displaymath}
		6-\sum_{i=1}^3 \max\{2-2\gamma(H_i'),2-\tilde{\gamma}(H_i')\} = \sum_{i=1}^3 \min\{2,1\} = 3.
	\end{displaymath}
	As a result, $G$ has a minor of genus and non-orientable genus at least~3, proving the claim.\qed
\end{proof}

We have now seen that the star-like graphs constructed this way fulfil none of the requirements necessary to apply one of the previously existing results, except for the tree-width~3 and order~3 contraction graph results.
But as soon as one of the non-planar graphs contains a $K_5$-minor, then the tree-width is at least~4.
So, by assuming that $H_1$ is the Petersen graph for example, we can assure that the graphs constructed here do not have the necessary tree-width.

To see that they do not admit a matching with a contraction graph of order~3, we show that $G$ has no 2-factor consisting of 3 cycles.
Let $F$ by any 2-factor of $G$, then $F$ contains a cycle $C_i$ completely contained in $G_i^-$ for all $i\in\{1,2,3\}$.
This is true as any cycle in $F$ is either completely contained in $G_i^-$, disjoint from it, or consists of a path in it.
But there can be at most one cycle that restricts to a path because $E(\V{G_i^-})$ contains only three edges.
Since this path is not Hamiltonian by Lemma~\ref{prop-of-gi-minus}~\ref{prop-of-gi-minus-Ham-paths}, $G_i^-$ contains at least one cycle of $F$.
Thus $F$ is made up of at least four cycles, one in each $G_i^-$ and one containing $x$.
This completes the proof that our examples fall in none of the classes covered previously.

\bibliographystyle{elsarticle-num}
\bibliography{lit}

\begin{thebibliography}{21}
\providecommand{\natexlab}[1]{#1}
\providecommand{\url}[1]{\texttt{#1}}
\providecommand{\urlprefix}{}

\bibitem[{Hoffmann-Ostenhof(2011)Arthur Hoffmann-Ostenhof}]{Ost11}
Hoffmann-Ostenhof A.
\newblock Nowhere-zero flows and structures in cubic graphs.
\newblock PhD thesis, University of Vienna; 2011.

\bibitem[{Cameron(2011)Cameron, Peter J.}]{Cam11}
Cameron PJ.
\newblock {Research Problems from the BCC22}.
\newblock Discrete Math 2011 Jul;311(13):1074--1083.
\newblock \urlprefix\url{https://doi.org/10.1016/j.disc.2011.02.024}.

\bibitem[{Akbari et~al.(2015)Akbari, Saieed and Jensen, Tommy R. and Siggers,
  Mark}]{AJS15}
Akbari S, Jensen TR, Siggers M.
\newblock Decompositions of Graphs into Trees, Forests, and Regular Subgraphs.
\newblock Discrete Math 2015 Aug;338(8):1322--1327.
\newblock \urlprefix\url{http://dx.doi.org/10.1016/j.disc.2015.02.021}.

\bibitem[{Abdolhosseini et~al.(2016)F. Abdolhosseini and S. Akbari and H.
  Hashemi and M. S. Moradian}]{AAHM16}
Abdolhosseini F, Akbari S, Hashemi H, Moradian MS, {Hoffmann-Ostenhof's
  conjecture for traceable cubic graphs}; 2016.
\newblock \urlprefix\url{https://arxiv.org/abs/1607.04768}.

\bibitem[{Ozeki and Ye(2016)Kenta Ozeki and Dong Ye}]{OY16}
Ozeki K, Ye D.
\newblock Decomposing plane cubic graphs.
\newblock European Journal of Combinatorics 2016;52:40--46.
\newblock
  \urlprefix\url{http://www.sciencedirect.com/science/article/pii/S0195669815001924}.

\bibitem[{Bachstein(2015)Anna Caroline Bachstein}]{Bac15}
Bachstein AC, {Decomposition of Cubic Graphs on the Torus and Klein Bottle};
  2015.

\bibitem[{Hoffmann-Ostenhof et~al.(2018)Hoffmann-Ostenhof, Arthur and Kaiser,
  Tomáš and Ozeki, Kenta}]{HKO18}
Hoffmann-Ostenhof A, Kaiser T, Ozeki K.
\newblock Decomposing planar cubic graphs.
\newblock Journal of Graph Theory 2018;88(4):631--640.
\newblock
  \urlprefix\url{https://onlinelibrary.wiley.com/doi/abs/10.1002/jgt.22234}.

\bibitem[{Aboomahigir et~al.(2018)Elham Aboomahigir and Milad Ahanjideh and
  Saieed Akbari}]{AAA18}
Aboomahigir E, Ahanjideh M, Akbari S, Decomposing Claw-free Subcubic Graphs and
  $4$-Chordal Subcubic Graphs; 2018.
\newblock \urlprefix\url{https://arxiv.org/abs/1806.11009}.

\bibitem[{Lyngsie and Merker(2019)Lyngsie, {Kasper Szabo} and Martin
  Merker}]{LM19}
Lyngsie K, Merker M.
\newblock Decomposing graphs into a spanning tree, an even graph, and a star
  forest.
\newblock The Electronic Journal of Combinatorics 2019;26(1).

\bibitem[{Heinrich(2019)Irene Heinrich}]{Hei19}
Heinrich I.
\newblock On Graph Decomposition: Hajós' Conjecture, the Clustering
  Coefficient and Dominating Sets.
\newblock PhD thesis, Technische Universität Kaiserslautern; 2019.

\bibitem[{Xie et~al.(2020)Mengmeng Xie and Chuixiang Zhou and Shun
  Zhou}]{XZZ20}
Xie M, Zhou C, Zhou S.
\newblock Decomposition of cubic graphs with a 2-factor consisting of three
  cycles.
\newblock Discrete Mathematics 2020;343(6):111839.
\newblock
  \urlprefix\url{http://www.sciencedirect.com/science/article/pii/S0012365X20300303}.

\bibitem[{Petersen(1891)Petersen, Julius}]{Pet1891}
Petersen J.
\newblock Die {T}heorie der regul{\"a}ren graphs.
\newblock Acta Mathematica 1891 Dec;15(1):193.
\newblock \urlprefix\url{https://doi.org/10.1007/BF02392606}.

\bibitem[{Jaeger(1985)Jaeger, Fran{\c{c}}ois}]{Jae85}
Jaeger F.
\newblock A survey of the cycle double cover conjecture.
\newblock In: North-Holland Mathematics Studies, vol. 115 Elsevier; 1985.p.
  1--12.

\bibitem[{Diestel(2010)Diestel, Reinhard}]{Die10}
Diestel R.
\newblock Graph Theory, vol. 173 of Graduate Texts in Mathematics.
\newblock Fourth ed. Heidelberg; New York: Springer; 2010.

\bibitem[{Frick and Singleton(2007)Marietjie Frick and Joy Singleton}]{FS07}
Frick M, Singleton J.
\newblock Cubic maximal nontraceable graphs.
\newblock Discrete Mathematics 2007;307(7):885--891.
\newblock
  \urlprefix\url{http://www.sciencedirect.com/science/article/pii/S0012365X06005796},
  cycles and Colourings 2003.

\bibitem[{Holton and Sheehan(1993)Holton, Derek Allan and Sheehan, John}]{HS93}
Holton DA, Sheehan J.
\newblock The Petersen Graph, vol.~7.
\newblock Cambridge University Press; 1993.

\bibitem[{Clark and Entringer(1983)L. Clark and R. Entringer}]{CE83}
Clark L, Entringer R.
\newblock Smallest maximally nonhamiltonian graphs.
\newblock Periodica Mathematica Hungarica 1983;14(1):57 -- 68.
\newblock
  \urlprefix\url{https://akjournals.com/view/journals/10998/14/1/article-p57.xml}.

\bibitem[{Isaacs(1975)Rufus Isaacs}]{Isa75}
Isaacs R.
\newblock Infinite Families of Nontrivial Trivalent Graphs Which are not Tait
  Colorable.
\newblock The American Mathematical Monthly 1975;82(3):221--239.
\newblock \urlprefix\url{https://doi.org/10.1080/00029890.1975.11993805}.

\bibitem[{Wormald(1979)Wormald, Nicholas C.}]{Wor79}
Wormald NC.
\newblock Classifying K-connected cubic graphs.
\newblock In: Horadam AF, Wallis WD, editors. Combinatorial Mathematics VI
  Berlin, Heidelberg: Springer Berlin Heidelberg; 1979. p. 199--206.

\bibitem[{Battle et~al.(1962)Battle, Joseph and Harary, Frank and Kodama,
  Yukihiro and Youngs, J. W. T.}]{BHKY62}
Battle J, Harary F, Kodama Y, Youngs JWT.
\newblock Additivity of the genus of a graph.
\newblock Bull Amer Math Soc 1962;68:565--568.
\newblock \urlprefix\url{https://doi.org/10.1090/S0002-9904-1962-10847-7}.

\bibitem[{Stahl and Beineke(1977)Stahl, Saul and Beineke, Lowell W}]{BS77}
Stahl S, Beineke LW.
\newblock Blocks and the nonorientable genus of graphs.
\newblock Journal of Graph Theory 1977;1(1):75--78.

\end{thebibliography}

\end{document}